\documentclass[a4paper]{amsart}
\usepackage[margin=1in]{geometry}
\usepackage{amssymb}
\usepackage{amsmath}
\usepackage{amsthm}
\usepackage{amsmath,amscd}
\usepackage{mathrsfs}
\usepackage{amsmath,amscd}
\usepackage{tikz-cd}
\usepackage{hyperref}
\usepackage{xcolor}
\usepackage[all,cmtip,color,matrix,arrow]{xy}

\theoremstyle{plain}
\newtheorem{theo}{Theorem}[section]

\theoremstyle{theorem}
\newtheorem{lem}[theo]{Lemma}
\newtheorem{sublemma}[theo]{Sublemma}
\newtheorem{Prop}[theo]{\bf Proposition}
\newtheorem{Coro}[theo]{\bf Corollary}

\theoremstyle{definition}
\newtheorem{remark}[theo]{Remark}
\newtheorem{Exam}[theo]{\bf Example}
\newtheorem{Def}[theo]{Definition}

\newcommand{\rk}{\mathrm{rk}}
\newcommand{\I}{\mathbf{I}}

\newcommand{\PS}{\mathbb{P}}
\newcommand{\SO}{\mathcal{O}}

\newcommand{\Spec}{\mathrm{Spec}}

\newcommand{\Pic}{\mathrm{Pic}}
\newcommand{\U}{N}
\newcommand{\LC}{\mathcal{L}}
\newcommand{\EC}{E}

\newcommand{\tot}{\mathrm{tot}}
\newcommand{\OG}{OG_\pm}
\newcommand{\OGo}{OG_\pm^{E_1}}
\newcommand{\Ln}{L^\pm_n}
\newcommand{\ttLn}{\tilde{L}^\pm_n}
\newcommand{\tLn}{\tilde{L}^\mp_n}
\newcommand{\OGm}{OG_\mp^{E_1}}

\author{Thomas Hudson}
\address{Thomas Hudson, College of Transdisciplinary Studies, DGIST, Daegu, 42988,
	Republic of Korea}
\email{hudson@dgist.ac.kr }

\author{Arthur Martirosian}
\address{Arthur Martirosian, Heinrich-Heine-Universit\"{a}t D\"{u}sseldorf, Universit\"{a}tsstra\textnormal{\ss}e 1, 40225 Düsseldorf, Germany}
\email{arthur.martirosian@hhu.de }

\author{Heng Xie}
\address{Heng Xie, School of Mathematics, Sun Yat-sen University, Xingangxilu 135, Guangzhou, 510275, China}
\email{xieh59@mail.sysu.edu.cn}
\thanks{MSC classes: 11E81, 14M15, 19G99, and 14C20.}

\title{Witt groups of Spinor varieties}

\begin{document}
	\begin{abstract}
		We show that Witt groups of spinor varieties (aka.\ maximal isotropic Grassmannians) can be presented by combinatorial objects called ``even shifted young diagrams". Our method relies on the Blow-up setup of Balmer-Calm\`es, and we investigate the connecting homomorphism of the localization sequence via the projective bundle formula of Walter-Nenashev, the projection formula of Calm\`es-Hornbostel and the excess intersection formula of Fasel.  
	\end{abstract}
	\maketitle

	\section{Introduction}
	
	In the 1930s, Witt \cite{Witt} introduced a group structure on the set of isometry classes of quadratic forms over an arbitrary field, which is now known as the Witt group; for a nice survey article, see \cite{Balmer}.\ Witt groups give rise to a very interesting cohomology theory in algebraic geometry. Similarly to  the oriented cohomology theories  in the sense of Levine-Morel \cite{LM07} or Panin \cite{Panin} (e.g. K-theory and Chow groups), Witt theory also has a localization sequence, cf.\;\cite{Balmer}. However,  unlike what happens with these functors, in Witt theory pushforwards always keep track of the orientation and the relative codimension, a fact which makes computations via the localization sequence much trickier, cf.\;\cite{CH11}. To this day, not many computations of Witt groups of elementary projective schemes have been performed:  quadrics (cf.\;\cite{Ne} and \cite{Xie20}), projective bundles (cf.\;\cite{Ne} and \cite{Walter}, see also \cite{KSW20} and \cite{Rohrbach}), Grassmann varieties (cf.\;\cite{BC12}), curves and surfaces (cf.\;\cite{Zib14}), cellular varieties over  algebraically closed fields (cf.\;\cite{Zib11}), and real varieties (cf.\;\cite{KW16} and \cite{KW20}).
	
	In \cite{BC12}, Balmer--Calm\`es adopt an innovative approach known as the ``Blow-up setup" from \cite{BC}. This approach is effective, because it interprets the abstract connecting homomorphism arising from  the localization sequence of Witt groups in purely geometric terms. This key idea  motivates the current article and, in its vein, we study Witt groups of even  maximal isotropic Grassmannians $OG_+(n,E)$, \textit{maximal isotropic Grassmannians} for short. Provided that the ambient space $E$ has even dimension and it comes equipped with a non-degenerate symmetric bilinear form, these spaces are defined as subschemes of the usual Grassmannians $Gr(n,E)$: maximal isotropic Grassmannians parametrise those subspaces on which the symmetric form vanishes identically. 
	
	Before stating the main result (cf. Theorem \ref{thm:main-theorem}), we recall that in  \cite{BC12} Balmer--Calm\`es identified the additive generators of the Witt groups of Grassmannians and introduced a combinatorial object known as \textit{even Young diagrams}, which they used as an indexing set. Roughly speaking they consider a subfamily consisting of those Young diagrams whose inner edges (\textit{i.e.}\;those which do not lie on the outer rectangular frame) have even length. 
	
	In similar fashion, for maximal isotropic Grassmannians the additive generators can be described by means of \textit{even shifted Young diagrams}, where shifted Young diagrams are the combinatorial object used to index the generators of the Chow ring of $OG_+(n,E)$. In this case the outer frame consists of an upside-down staircase corresponding to the maximal shifted partition $\mu=(n-1, n-2,\cdots, 1)$ right-justified. As before, the even diagrams are those whose inner segments have even length and we denote by $\mathfrak{E}_{n-1}$ the set of even shifted Young diagram which are contained inside $\mu$. Example \ref{exam:og(7)} provides the full list of the even shifted Young diagrams contained in $\mu$ for $n=7$, these are exactly those whose internal segments have even length. 
	
	%$$
	%\begin{tikzpicture}[scale=0.5]
	%\draw (-2,2) -- (5,2);
	%\draw (-2,1) -- (5,1);
	%\draw (-1,0) -- (5,0);
	%\draw (0,-1) -- (5,-1);
	%\draw (1,-2) -- (5,-2);
	%\draw (2,-3)-- (5,-3);
	%\draw (3,-4) -- (5,-4);
	%\draw (4,-5) -- (5,-5);
	%\draw (-2,2) -- (-2,1);
	%\draw (-1,2) -- (-1,0);
	%\draw (0,2) -- (0,-1);
	%\draw (1,2) -- (1,-2);
	%\draw (2,2) -- (2,-3);
	%\draw (3,2) -- (3,-4);
	%\draw (4,2) -- (4,-5);
	%\draw (5,2) -- (5,-5);   %%frame
	
	%\draw[fill, opacity=0.2] (-2,2) -- (5,2) -- (5,1)-- (3,1) -- (3,-1) -- (0,-1) -- (0,0) -- (-1,0) -- (-1,1) -- (-2,1) -- (-2,2);  %%%fill % -- (2,-2) -- (1,-2) -- (1,-1) --(0,-1)
	
	%\draw (9,0) node[right]{$\underline{\lambda}=(7,4,3)$};% \in \Delta_7$};
%\draw (12,-1.5) node[right] {$\underline{d}=(1,3,4)$};
%\draw (12,-3) node[right] {$\underline{e}=(0,2,3)$};
%\draw (-3,1.5) node{$\lambda_1 \rightarrow$};
%\draw (-2,0.5) node{$\lambda_2 \rightarrow$};
%\draw (-1,-0.5) node{$\lambda_3 \rightarrow$};
%<,2);
%\draw (8.75,2) -- (9,2);
%\draw (6.75,1) -- (7,1);
%\draw (7.75,-1) -- (8,-1);
%\draw (8.75,-2) -- (9,-2);  %%horizontal d-lines

%\draw (3,-7) -- (5,-7);
%\draw (2,-8) -- (5,-8);   %%horizintal e-lines

%\draw (3,-6.75) -- (3,-7);
%\draw(5,-6.75) -- (5,-7);
%\draw (2,-7.75) -- (2,-8);
%\draw (5,-7.75) -- (5,-8);
%\draw (5,-5.75) -- (5,-6);
%\draw (5,-6) node[below] {$e_1$};
%\draw (4.5,-7) node[below] {$e_2$};
%\draw (4,-8) node[below] {$e_3$};

%\draw[very thick] (5,1) -- (3,1) -- (3,-1) -- (1,-1) ;%-- (2,-2);
%\end{tikzpicture}
%$$

Let $S$ be a scheme with $\frac{1}{2} \in \SO_S$. Let $X$ be a scheme over $S$. Define the total Witt ring as
\[W^\tot(X) : = \bigoplus_{i \in \mathbb{Z}/4\mathbb{Z}} \bigoplus_{[L] \in \Pic(S)/2} W^i(X, p^*L),\]
where $p:X \to S$ is the structure morphism. Note that this convention differs from \cite{BC12}, as we do not adopt the whole grading involving $\Pic(X)/2$. This simplification is enough for us, because in our case the twisted Witt groups are trivial (cf. Proposition \ref{thm:L2d-1}).\ We can now state the main theorem.

\begin{theo}\label{thm:main-theorem}
	Let $S$ be a regular noetherian scheme with $\frac{1}{2} \in \SO_S$.\ Let $OG_+(n,E)$ be the maximal  isotropic Grassmannian of the trivial $2n$-dimensional bundle $E=\SO_S^{2n}$ with the split hyperbolic form (In fact, we deal with a more general setting of complete flags throughout the paper).\ There is an isomorphism of graded $W^\tot(S)$-modules
	$$W^{\tot}(OG_+(n,E)) \cong \bigoplus_{\underline{\lambda} \in \mathfrak{E}_{n-1}} W^{\tot}(S)[-\lvert \underline{\lambda}\rvert]\,. $$
	Here $\lvert \underline{\lambda} \rvert$ denotes the number of boxes of the even shifted Young diagram $\underline{\lambda}$ or, equivalently, the sum of all entries in the strict partition associated to $\underline{\lambda}$.	
\end{theo}

\begin{Exam}\label{exam:og(7)}
	Here are the even shifted Young diagrams for $\mathrm{OG}_+(7,E)$.
	
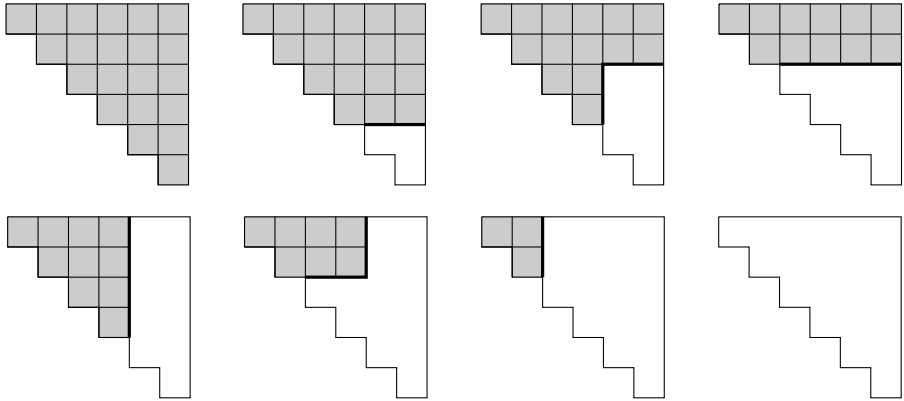
	$$
	\begin{tikzpicture}[scale=0.4]
		\draw (1,0) -- (7,0) -- (7,-6) -- (6,-6) -- (6,-5) -- (5,-5) -- (5,-4) -- (4,-4) -- (4,-3) -- (3,-3) -- (3,-2) -- (2,-2) -- (2,-1) -- (1,-1) -- (1,0);
		\draw (1,-1) -- (7,-1);\draw (2,-2) -- (7,-2);\draw (3,-3) -- (7,-3);
		\draw (4,-4) -- (7,-4);\draw (5,-5) -- (7,-5);
		\draw (2,0) -- (2,-2);\draw (3,0) -- (3,-3);\draw (4,0) -- (4,-4);
		\draw (5,0) -- (5,-5);\draw (6,0) -- (6,-6);
		\draw[fill,opacity=0.2] (1,0) -- (7,0) -- (7,-6) -- (6,-6) -- (6,-5) -- (5,-5) -- (5,-4) -- (4,-4) -- (4,-3) -- (3,-3) -- (3,-2) -- (2,-2) -- (2,-1) -- (1,-1) -- (1,0);
		%\draw (4,-7) node{$\rho \colon \{0,1,2\} \mapsto -21$};
	\end{tikzpicture}
	\qquad
	\begin{tikzpicture}[scale=0.4]
		\draw (1,0) -- (7,0) -- (7,-6) -- (6,-6) -- (6,-5) -- (5,-5) -- (5,-4) -- (4,-4) -- (4,-3) -- (3,-3) -- (3,-2) -- (2,-2) -- (2,-1) -- (1,-1) -- (1,0);
		\draw (7,-4) -- (5,-4);
		\draw (1,-1) -- (7,-1);\draw (2,-2) -- (7,-2);\draw (3,-3) -- (7,-3);
		\draw (4,-4) -- (7,-4);
		\draw (2,0) -- (2,-2);\draw (3,0) -- (3,-3);\draw (4,0) -- (4,-4);
		\draw (5,0) -- (5,-4);\draw (6,0) -- (6,-4);
		\draw[fill,opacity=0.2] (1,0) -- (7,0) -- (7,-4) -- (5,-4) -- (4,-4) -- (4,-3) -- (3,-3) -- (3,-2) -- (2,-2) -- (2,-1) -- (1,-1) -- (1,0);
		%\draw (4,-7) node{$ \rho \colon\{0,1\}  \mapsto -18$};
		\draw[very thick] (7,-4) -- (5,-4);
	\end{tikzpicture}
	\qquad
	\begin{tikzpicture}[scale=0.4]
		\draw (1,0) -- (7,0) -- (7,-6) -- (6,-6) -- (6,-5) -- (5,-5) -- (5,-4) -- (4,-4) -- (4,-3) -- (3,-3) -- (3,-2) -- (2,-2) -- (2,-1) -- (1,-1) -- (1,0);
		\draw (7,-2) -- (5,-2) -- (5,-4);
		\draw (1,-1) -- (7,-1);\draw (2,-2) -- (7,-2);\draw (3,-3) -- (5,-3);
		\draw (4,-4) -- (5,-4);
		\draw (2,0) -- (2,-2);\draw (3,0) -- (3,-3);\draw (4,0) -- (4,-4);
		\draw (5,0) -- (5,-2);\draw (6,0) -- (6,-2);
		\draw[fill,opacity=0.2] (1,0) -- (7,0) -- (7,-6) -- (7,-2) -- (5,-2) --(5,-4) -- (4,-4) -- (4,-3) -- (3,-3) -- (3,-2) -- (2,-2) -- (2,-1) -- (1,-1) -- (1,0);
		%\draw (4,-7) node{$\rho \colon  \{0,2\}\mapsto -14$};
		\draw[very thick]  (5,-4)--(5,-2) --(7,-2);
	\end{tikzpicture}
	\qquad
	\begin{tikzpicture}[scale=0.4]
		\draw (1,0) -- (7,0) -- (7,-6) -- (6,-6) -- (6,-5) -- (5,-5) -- (5,-4) -- (4,-4) -- (4,-3) -- (3,-3) -- (3,-2) -- (2,-2) -- (2,-1) -- (1,-1) -- (1,0);
		\draw (7,-2) -- (3,-2);
		\draw (1,-1) -- (7,-1);
		\draw (2,0) -- (2,-2);\draw (3,0) -- (3,-2);\draw (4,0) -- (4,-2);
		\draw (5,0) -- (5,-2);\draw (6,0) -- (6,-2);
		\draw[fill,opacity=0.2] (1,0) -- (7,0) -- (7,-2) -- (3,-2) -- (2,-2) -- (2,-1) -- (1,-1) -- (1,0);
		%\draw (4,-7) node{$\rho \colon\{0\} \mapsto -11$};
		\draw[very thick]  (3,-2)--(7,-2);
	\end{tikzpicture}
	$$
	$$
	\begin{tikzpicture}[scale=0.4]
		\draw (1,0) -- (7,0) -- (7,-6) -- (6,-6) -- (6,-5) -- (5,-5) -- (5,-4) -- (4,-4) -- (4,-3) -- (3,-3) -- (3,-2) -- (2,-2) -- (2,-1) -- (1,-1) -- (1,0);
		\draw (5,0) -- (5,-4);
		\draw (1,-1) -- (5,-1);\draw (2,-2) -- (5,-2);\draw (3,-3) -- (5,-3);
		\draw (4,-4) -- (5,-4);
		\draw (2,0) -- (2,-2);\draw (3,0) -- (3,-3);\draw (4,0) -- (4,-4);
		\draw[fill,opacity=0.2] (1,0) -- (5,0) -- (5,-4) -- (4,-4) -- (4,-3) -- (3,-3)-- (3,-2)  -- (2,-2) -- (2,-1) -- (1,-1) -- (1,0);
		\draw[very thick]  (5,-4)--(5,0);
		%\draw (4,-7) node{$\rho \colon\{1,2\} \mapsto-10$};
	\end{tikzpicture}
	\qquad
	\begin{tikzpicture}[scale=0.4]
		\draw (1,0) -- (7,0) -- (7,-6) -- (6,-6) -- (6,-5) -- (5,-5) -- (5,-4) -- (4,-4) -- (4,-3) -- (3,-3) -- (3,-2) -- (2,-2) -- (2,-1) -- (1,-1) -- (1,0);
		\draw (5,0) -- (5,-2) -- (3,-2);
		\draw (1,-1) -- (5,-1);
		\draw (2,0) -- (2,-2);\draw (3,0) -- (3,-2);\draw (4,0) -- (4,-2);
		\draw[very thick]  (3,-2)--(5,-2)--(5,0);
		\draw[fill,opacity=0.2] (1,0) -- (5,0) -- (5,-2) -- (3,-2)  -- (2,-2) -- (2,-1) -- (1,-1) -- (1,0);
		%\draw (4,-7) node{$\rho \colon\{1\} \mapsto -7$};
	\end{tikzpicture}
	\qquad
	\begin{tikzpicture}[scale=0.4]
		\draw (1,0) -- (7,0) -- (7,-6) -- (6,-6) -- (6,-5) -- (5,-5) -- (5,-4) -- (4,-4) -- (4,-3) -- (3,-3) -- (3,-2) -- (2,-2) -- (2,-1) -- (1,-1) -- (1,0);
		\draw (3,0) -- (3,-2);
		\draw (1,-1) -- (3,-1);
		\draw (2,0) -- (2,-2);
		\draw[very thick]  (3,-2)--(3,0);
		\draw[fill,opacity=0.2] (1,0) -- (3,0) -- (3,-2) -- (2,-2) -- (2,-1) -- (1,-1) -- (1,0);
		%\draw (4,-7) node{$\rho \colon\{2\} \mapsto -3$};
	\end{tikzpicture}
	\qquad
	\begin{tikzpicture}[scale=0.4]
		\draw (1,0) -- (7,0) -- (7,-6) -- (6,-6) -- (6,-5) -- (5,-5) -- (5,-4) -- (4,-4) -- (4,-3) -- (3,-3) -- (3,-2) -- (2,-2) -- (2,-1) -- (1,-1) -- (1,0);
		%\draw (4,-7) node{$\rho \colon\emptyset \mapsto 0$};
	\end{tikzpicture}
	$$
\end{Exam}

It is worth pointing out that, although we work within the framework of the Blow-up setup of \cite{BC}, our method differs from that of \cite{BC12} for Type A. There the authors make use of desingularization of Schubert varieties as an intermediate step to compute the connecting homomorphism $\partial$, instead, in view of the observation that the exceptional fiber of the Blow up is a projective bundle, we are able to handle $\partial$ in Type D using the projective bundle formula and the excess intersection formula for regular schemes.\ This idea does not seem to have appeared in the literature so far.\ Since the connecting homomorphisms are not always trivial, the localization sequence can not split in general. Notice, however, that in a specific case (cf.\ Theorem \ref{thm:odd dimension case})  the non-split localization sequence can be transformed into short split exact sequences, provided that one extracts the data from the codimension two subbundle of the ambient bundle. This explains why even Young diagrams come into the picture. It would be interesting to know if this method can be used to compute the Witt groups of other homogeneous varieties. Unfortunately, we were informed by Nicolas Perrin that Type C and E do not fit in the Blow-up Setup, and therefore it can not be applied directly to these cases without modifications.

One may also try to adopt our method to study the $\I$-cohomology of Type D homogeneous varieties. For $\I$-cohomology of projective bundles and split quadrics using the Blow-up setup, we refer the reader to \cite{Fasel13} and \cite{HXZ20}. There are also recent developments on the Hermitian $K$-theory when 2 is not invertible in the base, cf. \cite{Calmeselt} and \cite{Sch21}. Investigating the Hermitian $K$-theory of schemes in these frameworks seems to be an interesting project.\\

\noindent \textbf{Convention}. All our schemes, if not mentioned otherwise, are assumed to be \textit{regular noetherian} with $\frac{1}{2}$ in their global sections. We refer to \cite{Balmer}, \cite{BC}, \cite{Kne77} and \cite{QSS} for basic terminology.

\section{Geometry of spinor varieties}

\subsection{Flags} 
Let $(\EC,\beta)$ be a bilinear space (aka.\ a non-degenerate symmetric bilinear bundle) of rank $2n$ over a scheme $S$ with $\frac{1}{2} \in \SO_S$. Throughout the article we assume that $S$ is connected, and the general case of our main theorem follows easily from this one.\ 

\begin{Def}
	The bilinear space $\EC = (\EC,\beta)$ is said to \textit{admit a complete flag} if there is a filtration 
	\[ \EC_\bullet : = \Big( 0 = \EC_0 \subset \EC_1 \subset \EC_2 \subset \cdots \subset \EC_{n-1} \subset \EC_n = \EC_n^\perp \subset \EC_{n-1}^\perp  \subset \cdots \subset \EC_1^\perp \subset \EC_0^\perp = \EC \Big) \] 
with the rank $\rk(\EC_i) = i$ and all the inclusions are admissible (\textit{i.e.} their quotients are locally free). 
\end{Def}  
In particular, $(\EC,\beta)$ is metabolic, and $E_n$ is a \textit{Lagrangian}, \textit{i.e.} \ a maximal totally isotropic  subbundle.\ The symmetric bundle $(\EC,\beta) $ induces a unique symmetric bundle $(\EC^1,\beta ^1)$ with $\EC^1: = \EC_1^\perp/\EC_1$, see \cite[Theorem 1.1.32]{Balmer}. Note that the complete flag $\EC_\bullet$ induces a complete flag 
\[ \EC^1_\bullet : = \Big( 0 = \EC^1_{0}  \subset \EC^1_{1} \subset \cdots \subset \EC^1_{n-2} \subset \EC^1_{n-1} = (\EC^1_{n-1})^\perp \subset (\EC^1_{n-2})^\perp  \subset \cdots \subset (\EC^1_{0})^\perp =: (\EC^1)^\perp  \Big) \] 
on  $(\EC^1,\beta ^1)$, where we define $\EC^1_{i}: = \EC_{i+1}/\EC_1$. Notice that $\EC_{i+1}^\perp / \EC_1$ is isomorphic to $ (\EC_{i+1}/\EC_1)^{\perp}$ in $(\EC^1,\beta ^1)$, cf.\ \cite[Proposition 6.5]{QSS}.\ This procedure can be repeated and inductively one obtains complete flags $\EC^j_\bullet$ on  $(\EC^j,\beta ^j)$, where $\EC^j : = ((\EC^{j-1}_1)^\perp/\EC^{j-1}_1,\beta ^{j-1})$ and $\EC^0_1 = \EC_1$.\ Each $\EC^j$ is a metabolic space with a Lagrangian $E^j_{n-j}$. 

\begin{Prop}\label{Prop: rank two} Let $E$ be a metabolic space of rank $2$. Then, $E$ has exactly two Lagrangians $N$ and $N'$. Moreover, $N' \cong N^\vee$ and $E \cong N \oplus N^\vee$. 
\end{Prop}
\begin{proof}
This is well-known if the base is smooth over a field, cf.\ \cite[page 77]{FP}. We include details, as we could not find a reference in the generality that we need.\ Assume that $N$ is a Lagrangian of $E$. Then, we have the structural exact sequence
$$ \xymatrix{ 0 \ar[r] & N \ar[r] & E \ar[r] & N^\vee \ar[r] & 0}. $$
Let $\{ U_i \}$ be an affine open cover of the base scheme so that $E|_{U_i}$ is free. Since $ \frac{1}{2} \in \SO_{U_i}$, recall that the metabolic space $E|_{U_i}$ of rank $2$ has exactly two Lagrangians  over each affine space  \cite{Kne77}. The first one is precisely $N|_{U_i}$ and we will denote the second one by $N'_i$. Note that $N'_i |_{U_i \cap U_j} = N'_j |_{U_i \cap U_j}$, since both of them are different from $N|_{U_i \cap U_j}$. This shows that the different $N'_i$'s glue together, giving rise to a new Lagrangian $N'$ of $E$. Next, take any Lagrangian $N''$ of $E$. If $N''|_{U_i} = N|_{U_i}$ for some $i$, then $N'' = N$. If not, then there is an affine open subscheme $V$ such that $N''|_{V} = N'|_{V}$. It follows that $N''|_{U_i \cap V} = N|_{U_i \cap V}  = N'|_{U_i \cap V} $ which is a contradiction. This shows that $N''$ is either $N$ or $N'$.\ Finally, the Lagrangian $N'$ gives rise to the following diagram
$$ \xymatrix{ 0 \ar[r] & N \ar[r] & E \ar[r] & N^\vee \ar[r] & 0 \\
                                  &           &  N' \ar[u] \ar[ur]  }. $$
Note that the composite $N' \to E \to N^\vee$ is locally an isomorphism, and hence an isomorphism. 
\end{proof}
\begin{remark} 
Although metabolic spaces of rank 2 are split, there exist examples of metabolic spaces of rank $4$ which do not split, cf. \cite{Knus-Ojanguren}. 
\end{remark}

\begin{lem}\label{lem:lift-sublag}
Assume that $V$ is totally isotropic. Take $W \subset V^\perp/V$ to be a Lagrangian. Denote by $L$ the pullback along the inclusion $W \hookrightarrow V^\perp/V$ and the canonical quotient $V^\perp \to V^\perp/V $. Then $L$ is a Lagrangian of $E$ such that $V\subset L$ and $L/V = W$.                      
$$\xymatrix{ W \ar@{>->}[r] & V^\perp/V \\
                     L \ar@{>->}[r]  \ar[u]& V^\perp \ar[u] \ar@{>->}[r] & E}$$
\end{lem}
\begin{proof}
See \cite[Proposition 6.5]{QSS}.
%Consider the diagram
%	\[\xymatrix{ W \ar@{>->}[r]   & V^\perp/V \ar@{->>}[r] & W^\vee   \\
%		L  \ar[u]^-{s} \ar@{>->}[r]    & \ar[u] V^\perp \ar@{->>}[r] & C \ar[u]^-{\cong} \\
%		V \ar[u]\ar@{=}[r] & V \ar[u] \ar[r] &\  0\ , \ar[u]}\]
%	where the left upper square is cartesian and $C$ is the cokernel of $L \to V^\perp$. By construction, the kernel of $s: L \to W$ is $V$. We wish to show that $L = L^\perp$ in $E$.\ To show this, we look at another diagram:
%	\[ \xymatrix{  L  \ar@{=}[d] \ar@{>->}[r]    & \ar@{>->}[d] V^\perp \ar[r] & W^\vee \ar[d]^-{s^\vee} \\
%		L \ar@{>->}[r]^-{j} & \EC \ar[r]^-{j^\vee \beta} & L^\vee.}\]
%	Note that the composition of the morphisms on the upper row is zero. This implies that in the bottom row, one has $j^\vee \beta j = 0$. To see the bottom row is exact, we use \cite[Chapter I.2, Proposition 1]{Kne77}. The result follows. 
\end{proof}

\begin{remark}\label{rmk:tildeE} Start with a complete flag $E_\bullet$. By Proposition \ref{Prop: rank two} and Lemma \ref{lem:lift-sublag}, we see that there exists a unique complete flag $\tilde{E}_\bullet$ such that $\tilde{E}_i = E_i$ for $i\leq n-1$ and for which $\tilde{E}_n/E_{n-1}, E_n/E_{n-1}$ form the two Lagrangians of $E^{n-1}$. Therefore, this verifies that \cite[page 77]{FP} applies to our situation. 
\end{remark}

\subsection{Isotropic Grassmann bundles}
Even without a symmetric structure on $\EC$, the Grassmannian scheme $Gr(d,\EC)$ can be defined on the functor of points as
\begin{align*} Gr(d,\EC)(X) & = \Big\{ \U \subset \EC_X : \EC_X/ \U \textnormal{ is a locally free $\SO_X$-module of rank $2n-d$} \Big\} \\
	& = \Big\{ \U \subset \EC_X : \U(x)  \hookrightarrow \EC_X(x) \textnormal{ is a $k(x)$-vector subspace of rank $d$,  $\forall x \in X$} \Big\},
 \end{align*}
where for every given scheme $X$ one defines $E_X := p^*E$ with the structure morphism $p: X \to S$ (cf. \cite[p. 210 -211]{GW}). Let $L_d$ be the universal bundle of $Gr(d,E)$. It fits into the exact sequence
	 $$ \xymatrix{ 0 \ar[r] & L_d \ar[r] & E_{Gr(d,\EC)} \ar[r] & Q_d \ar[r] & 0, }$$
 where $Q_d$ is called the universal quotient bundle. The tangent bundle of $Gr(d,E)$ can be identified with the bundle $Hom(L_d, Q_d)$ via a second fundamental form homomorphism, cf. \cite[Appendix B.5.8]{Fu}. 

\begin{Def}
	Define $OG(n,\EC)$ to be the subscheme of the Grassmannian $Gr(n,\EC)$ that parametrizes maximal totally isotropic  subbundles in $E$ with respect to the form $\beta$. More concretely, the subscheme $OG(n, \EC)$ is defined as the locus where
	the sequence 
	 \begin{equation}\label{eq:universal-exact-seq}
	 	\xymatrix{ 0 \ar[r] & L_n \ar[r]^-{i} & E_{Gr(n,\EC)} \ar[r]^-{i^\vee \circ \beta} & L_n^\vee \ar[r] & 0 }
	 \end{equation} 
	 on $Gr(n,\EC)$ is exact. Denote the correspondent embedding by $\kappa:OG(n,\EC)\hookrightarrow Gr(n,\EC)$. 
\end{Def}

\begin{lem}\label{lem:tangentOGn} The relative tangent bundle $T_{OG(n,E)/S}$ of $OG(n,E)$ over $S$ is isomorphic to $ \wedge^2( \kappa^* L_n^\vee)$. 
\end{lem}
\begin{proof}
Note that $OG(n,\EC)$ can also be viewed as the zero locus of the regular section $\beta|_{L_n}: L_n \otimes L_n \to \SO$ in the symmetric power $S^2(L_n^\vee)$. As a consequence the normal bundle $N_\kappa$ of $OG(n,\EC)$ in $ Gr(n,\EC)$  is isomorphic to $S^2(\kappa^* L_n^\vee)$.\ The sequence (\ref{eq:universal-exact-seq}) becomes exact on $OG(n,\EC)$, one has $ \kappa^* Q_n\cong\kappa^*L_n^\vee $. Hence, 
$$\kappa^* T_{Gr(n,E)/S} = \kappa^* Hom(L_n, Q_n) = Hom( \kappa^*L_n,  \kappa^* Q_n) = \kappa^*L_n^\vee \otimes \kappa^*L_n^\vee \, .$$
 This implies that the exact sequence of bundles on $OG(n,E)$ from \cite[Appendix B.7]{Fu}
$$ \xymatrix{0 \ar[r] & T_{OG(n,E)/S} \ar[r] & \kappa^* T_{Gr(n,E)/S} \ar[r] & N_\kappa \ar[r] & 0 } $$
gets identified with the usual exact sequence
$$ \xymatrix{0 \ar[r] & \wedge^2 (\kappa^* L_n^\vee)  \ar[r] & \kappa^* L_n^\vee \otimes \kappa^* L_n^\vee \ar[r] & S^2(\kappa^* L_n^\vee) \ar[r] & 0 \,, } $$
 showing that $T_{OG(n,E)/S} \cong \wedge^2 (\kappa^* L_n^\vee) $. 
\end{proof}
\begin{remark}
For simplicity we will drop $\kappa^*$ from the notation and simply write $L_n$ instead of the more precise $\kappa^* L_n$ to refer to the pullback bundle to $OG(n,E)$.  
\end{remark}
If $(E,\beta)$ is metabolic, then by base-change $(E_X,\beta_X)$ is also metabolic. On the functor of points, we have
	\[ OG(n,\EC)(X) : = \Big\{ N \in Gr(n,\EC)(X) : N^\perp = N  \Big\}. \]
Elements in $OG(n,\EC)(X)$ are precisely the Lagrangians of $(E_X, \beta_X)$. 
%Note that the condition $N^\perp = N $ means that the sequence
%$$ \xymatrix{ 0 \ar[r] & N \ar[r]^-{i} & E_X \ar[r]^-{ i^\vee \circ \beta_X} & N^\vee \ar[r] & 0} $$
%is exact where $i$ is the inclusion and $N^\vee$ is the dual bundle of $N$, cf.\ \cite{Balmer}.\ 
The following fact is well-known.  
\begin{lem}\label{lem: independent of points}
If $N$ and $N'$ are Lagrangians of $(E_X,\beta_X)$, then the function 
\begin{align*}
\Gamma: X &\to \mathbb{Z} /2\mathbb{Z}  \\
x &\mapsto \rk(N(x)\cap N'(x))
\end{align*}
 is constant on each connected component of $X$. 
\end{lem}
\begin{proof}
The proof basically follows from \cite[the proof of the theorem is on p.\ 184]{Mumford}. For the readers' convenience, we provide more details. We may assume that $X$ is connected. Let $U_i = \Spec \,A_i$ ($i \in I$) be a finite affine open cover of $X$ on which $N$ and $N'$ are both trivial.\ One can find isometries $\psi_i : (E_{U_i}, \beta_{U_i}) \xrightarrow{\cong}  \mathbb{H}(N_{U_i})$ and $ \phi_i : (E_{U_i}, \beta_{U_i}) \xrightarrow{\cong}  \mathbb{H}(N'_{U_i})  $, where $\mathbb{H}(N_{U_i})$ and  $\mathbb{H}(N'_{U_i})$ are hyperbolic spaces. Therefore, there exists an isometry  $\varphi_i : (E_{U_i}, \beta_{U_i}) \xrightarrow{\cong} (E_{U_i}, \beta_{U_i}) $ such that $\varphi_i(N_{U_i}) = N'_{U_i} $. Now, by the definition of isometry, we get the equality $ \beta_{U_i} = \varphi_{i}^\vee \beta_{U_i} \varphi_{i} $. By taking the determinant, we obtain $(\det \varphi_i)^2 = 1 \in A_i$. Since $U_i$ is irreducible, we see that $\det \varphi_i = \pm 1$. 

Note that, since $X$ is irreducible, %(why can we assume that X is irreducible?)
one has $U_i \cap U_j \neq \emptyset$.\ Take a point $s \in U_i \cap U_j$.\ By \cite[Exercise 18 (d) in Section 6]{Bourbaki}, we deduce that $\det \varphi_i(s) = \det \varphi_j(s) = (-1)^{n-q}$ over the residue field $k(s)$, where $q = \rk (N(s) \cap N'(s))$. Since $2$ is invertible, this implies that $$\det \varphi_i = \det \varphi_i(s) = \det \varphi_j(s) = \det \varphi_j = (-1)^{n-q}.$$ 
Therefore, $\Gamma$ is well-defined on each intersection $U_i \cap U_j$ and constant on each $U_i$. The result follows.
\end{proof}
\begin{remark}
For readers interested in more general settings, it would seem that the proof of Lemma \ref{lem: independent of points} generalises to integral noetherian schemes with $\frac{1}{2}$ in their global sections.
\end{remark}
In particular, Lemma \ref{lem: independent of points} implies that the scheme $OG(n,\EC)$ has two disjoint connected components $OG_+(n,\EC)$ and $OG_-(n,\EC)$ defined as
	\begin{align*}
		OG_\pm(n,\EC)(X) & := \Big\{ N \in OG(n, \EC)(X)  : \rk ( \EC_{n}(x) \cap N(x)) \equiv n_\pm \textnormal{ (mod } 2)  \textnormal{ for any } x \in X  \Big\}, 	\end{align*}
		where $n_+ = n$ and $n_- = n-1$. In the sequel we shall drop the mention of $X$ to avoid cumbersome notation. The connected component $OG_+(n,\EC)$ is usually called the \textit{spinor variety} (or \textit{maximal  isotropic Grassmannian}).
%\begin{remark}
%In this paper we only deal with the split hyperbolic case, as we do not know how to define the global isomorphism $\tau$ in the proof of Lemma \ref{lem:OG+OG-} if we work over complete flags. In other words, we do not know if the two connected components are still isomorphic to each other in the general setting. 
%\end{remark}
	
In view of Proposition \ref{Prop: rank two} it is clear that any metabolic space $E$ of rank $2$ admits a complete flag
$$ 0= E_0 \subset E_1 = E_1^\perp \subset E_2 = E. $$
Moreover, Proposition \ref{Prop: rank two} also implies the following result.
\begin{Coro}
If $E$ is metabolic of rank 2, then $OG_{\pm}(1,E) = S$. 
\end{Coro}

\subsection{The closed embedding}\label{sect:closed embeds}

Let $V$ be a totally isotropic subbundle of $\EC$. Consider the closed subscheme $\iota^V_\pm:OG_\pm^{V}(n,\EC) \hookrightarrow OG_\pm(n, \EC)$ given by
\begin{align*}
	OG_\pm^{V}(n,\EC) = \{N \in OG_\pm(n,\EC) : V \subset N \}.\end{align*}
In the special case $V =\EC_j$, we have the following result.
\begin{lem}\label{lem:embeddingPhi}
 The morphism
	\begin{align*}
		\Phi: OG_\pm^{\EC_j}(n,\EC) & \to OG_\pm(n-j, \EC^j) \\  
		\EC_j \subset N  \subset \EC &\mapsto N/\EC_j \subset \EC^j =  \EC_j^\perp/\EC_j
	\end{align*}
	is an isomorphism.
\end{lem}
\begin{proof}
	Set $N^j : = N/\EC_j$.\ By \cite[Proposition 6.5]{QSS}, we see that $(N^j)^{\perp_{\beta_j}} = N^j$.\ 
	%In order to verify that the requirement on the dimension is satisfied, we can assume that we are working over fields. Suppose that $\rk(N \cap \EC_{n} ) \equiv n \mod 2$, then we need to check that
	%\[ \rk(N^j \cap \EC_n^j) \equiv n-j \mod 2\]
	%holds.
	%We use the identity $N^j \cap \EC_n^j = (N \cap \EC_n)/\EC_j $. Thus, we have
	%\[ \rk (N^j \cap \EC_n^j ) = \rk((N \cap \EC_n)/\EC_j) = \rk (N \cap \EC_n) -j. \]
	%This shows that $\Phi$ is well-defined.
	To see that $\Phi$ is an isomorphism, we construct an inverse morphism
	$\Psi: OG_\pm(n-j, \EC^j) \to OG_\pm^{\EC_j}(n,\EC)  $
	by sending a Lagrangian $W \subset \EC^j$ to the Lagrangian $L$ of $\EC$ constructed as in Lemma \ref{lem:lift-sublag}. It is straightforward to check that $\Phi$ and $\Psi$ are inverses of each other. 
	\end{proof}
%A similar proof yields the following results.
%\begin{Coro}
%The morphism
%	\begin{align*}
%		\Phi: OG_\pm^{\EC_j}(n,\EC) &\to OG_\pm^{\EC_k^{j-k}}(n-(j-k),\EC^{j-k}) \\
%		\EC_j \subset N &\mapsto \EC_j/\EC_k \subset N/\EC_k
%	\end{align*}
%	is an isomorphism.
%\end{Coro}

\subsection{The Blow-up setup} In this subsection, we shall study the blow-up of the closed embedding 
$$\iota_\pm:OG^{E_1}_\pm(n,E) \hookrightarrow OG_\pm(n,E),$$ 
where $\iota_\pm : = \iota^{E_1}_\pm$ and we use this simplification if no confusion may occur.\ Let $L_n^\pm$ be the pullback of the universal bundle $L_n$ on $OG(n,E)$ along the canonical embedding $OG_\pm(n,E) \hookrightarrow OG(n,E)$. 
Consider the universal exact sequence
$$ \xymatrix{ 0 \ar[r] & L_n^\pm \ar[r] & E \ar[r] & (L_n^\pm)^\vee \ar[r] & 0}$$
on $OG_\pm(n,E)$. The subscheme $\OGo(n,E)$ of $\OG(n,E)$ is precisely the vanishing locus of the composite
$$\xymatrix{ & & E_1 \ar[d] \ar@{-->}[dr] \\
0 \ar[r] & \Ln \ar[r]  & E  \ar[r]  & (\Ln)^\vee \ar[r] & 0 }$$
or, equivalently, the vanishing locus of the composite 
$$\xymatrix{ 0 \ar[r] & \Ln \ar[r] \ar@{-->}[dr] & E  \ar[r] \ar[d] & (\Ln)^\vee \ar[r] & 0 \\
                                                & & (E_1)^\vee }$$
One can pullback the bundle $\Ln$ over $\OG(n,E)$ via the closed embedding $\iota_\pm: \OGo(n,E) \hookrightarrow \OG(n,E)$, which is  denoted by $\ttLn$. 
\begin{Def}
Define $Bl_\pm(n,E)$ to be the Grassmannian scheme $Gr(n-1, \tLn)$ over $\OGm(n,E)$, and denote by $\tilde{\alpha}_\mp:Bl_\pm(n,E)\to  \OGm(n,E) $ the projection. 
\end{Def}                                               
Let $P_{n-1}^\mp$ be the universal bundle of $Bl_\pm(n,E): = Gr(n-1, \tLn)$, and let $P^\mp_{n-1} \hookrightarrow \tLn$ be the canonical inclusion. Note that one has $E_1 \subset \tLn$ and $P^\mp_{n-1} \subset \tLn \subset E_1^\perp$.\ Consider the filtration
$$P^\mp_{n-1} \subset \tLn \subset (P^\mp_{n-1})^\perp \subset E $$
over $Bl_\pm(n,E)$.\ By taking the quotient, we obtain a Lagrangian $\tilde{T}^\mp: = \tLn /P^\mp_{n-1}$ inside the metabolic space $ (P^\mp_{n-1})^\perp/ P^\mp_{n-1}$ of rank $2$. By Proposition \ref{Prop: rank two}, we see that inside this metabolic space there is a unique Lagrangian $T^\pm$
% inside  the metabolic space $ (P^\mp_{n-1})^\perp/ P^\mp_{n-1}$
, which is different from $\tilde{T}^\mp$ and is in the other component.\ By Lemma \ref{lem:lift-sublag}, the bundle $T^\pm$ can be lifted to a Lagrangian $\pi L^\pm_n \subset E$ such that 
$$P^\mp_{n-1} \subset \pi L^\pm_n \subset (P^\mp_{n-1})^\perp \subset E. $$
 By the universal property of $\OG(n,E)$, we get a morphism
$$ \pi_\pm: Bl_\pm(n,E) \to \OG(n,E)  $$
from the Lagrangian $\pi L^\pm_n \subset E$ over $Bl_\pm(n,E)$ such that $\pi L^\pm_n = \pi^*_\pm \Ln$. Thus, we just adopt a simplified notation and write $\Ln$ to refer to the pullback $\pi L^\pm_n$ over $Bl_\pm(n,E)$. 
\begin{Def}
Define $E_\pm(n,E)$ to be the vanishing locus of the composite
$$\xymatrix{  & E_1  \ar[d] \ar@{-->}[dr]\\
P^\mp_{n-1}   \ar@{>->}[r] & \tilde{L}^\mp_n  \ar[r]& \tilde{L}^\mp_n/ P^\mp_{n-1}  }$$
over $Bl_\pm(n,E)$. Define $\tilde{\iota}_\pm: E_\pm(n,E) \hookrightarrow Bl_\pm(n,E)$ to be the closed embedding, and $\tilde{v}_\pm: U_\pm(n,E) \hookrightarrow Bl_\pm(n,E)$ to be its open complement.\  Define $v_\pm:= \pi_\pm \circ  \tilde{v}_\pm$ and $\alpha_\mp: = \tilde{\alpha}_\mp \circ  \tilde{v}_\pm$.
\end{Def}
Note that $P_{n-1}^\mp = \Ln \cap \tLn$, and $E_1\subset P_{n-1}^\mp $ if and only if $E_1 \subset \Ln$. Therefore, $E_\pm(n,E) $ can be identified with the fibre product $ \pi^{-1}_\pm OG_\pm^{E_1}(n,E)$, where the pullback morphism
$$ \tilde{\pi}_\pm: E_\pm(n,E) \to \OGo(n,E)  $$
classifies $ E_1 \subset P_{n-1}^\pm  \subset L_n^\pm$ over $E_\pm(n,E)$ by the universal property of $\OGo(n,E)$. Note also that $E_+(n,E) = E_-(n,E)$ by construction, and $\tilde{\alpha}_\mp \circ \tilde{\iota}_\pm = \tilde{\pi}_\mp$. Therefore, we only write $E(n,E)$ to denote $E_\pm(n,E)$. The following result is proved in \cite{Perrin} when the base is $\mathbb{C}$. 

\begin{theo}\label{theo:blowup}
 The scheme $Bl_\pm(n,E)$ is the blow-up of $OG_\pm(n,E)$ along $OG^{E_1}_\pm(n,E)$   with the exceptional fiber $E(n,E)$, and $ v_\pm: U_\pm(n,E) \to \OG(n,E)$ is the open complement of $OG^{E_1}_\pm(n,E)$ inside $OG_\pm(n,E)$.\ Moreover, the morphism $\alpha_\mp : U_\pm(n,E) \to OG_\mp(n,E)$ is an affine bundle.
\end{theo}
 
	 The situation is depicted in the following diagram
	\begin{equation}\label{eq:keydiagram}
		\xymatrix{OG^{\EC_1}_\pm(n,E)  \ar[rr]^-{\iota_\pm} && OG_\pm(n,E)   && \ar[ll]_-{v_\pm} U_\pm(n,E) \ar[d]^-{\alpha_\mp} \ar[lld]_-{\tilde{v}_\pm} \\
		%%%%%%%%%%%%%%%%
			E(n,E)  \ar[rr]^-{\tilde{\iota}_\pm} \ar[u]^-{\tilde{\pi}_\pm} \ar@/_3pc/[rrrr]_-{\tilde{\pi}_\mp} && Bl_\pm(n,E) \ar[u]^-{\pi_\pm} \ar[rr]^-{\tilde{\alpha}_\mp} && OG^{\EC_1}_\mp(n,E) }
	\end{equation} 
and the Blow-up setup (cf. \cite[Setup 1.1]{BC}) is therefore satisfied by Theorem \ref{theo:blowup}.\ When viewed in terms of functors of points, (\ref{eq:keydiagram}) has the following interpretation:
	\[  \xymatrixcolsep{0.5pc}
	\begin{small}
	\xymatrix{
		 \Big\{L_n^\pm \ |  L_n^\pm \supset E_1\Big\}
		\ar[r] 
		&\Big\{L_n^\pm \ | E \supset L_n^\pm  \Big\}
		%\ar[r] 
		&& \Big\{L_n^\pm\ |\ L_n^\pm\not\supset E_1\Big\}
		\ar[dll]\ar[d]\ar[ll] 		\\
		% % % % % % % % % % % % % % % % % % % % % % % % % % % % % % % % % % % % % % % % % % % % % % % % % % % % % % % % % %
		\left\{(P_{n-1},L_n^\pm,L_n^\mp)\ \middle|\: \
		{\begin{matrix} 
		L_n^\pm \supset P_{n-1} 
	\\ L_n^\mp \supset P_{n-1} 
\\P_{n-1} \supset E_1
	\end{matrix}}   \right\} 
		\ar[r]\ar[u] 
		&\left\{(P_{n-1},L_n^\pm, L_n^\mp )\ \middle|\: \ 
		{\begin{matrix} L_n^\pm \supset P_{n-1} \\ L_n^\mp \supset P_{n-1}\\ L_n^\mp\supset E_1 
			\end{matrix}} \right\}
		\ar[rr]\ar[u]
		&& \left\{ L_n^\mp \ |\ 
		{\begin{matrix}L_n^\mp \supset E_1 
		 	\end{matrix}}  \right\}.	}
	\end{small} \] 

\begin{proof}
By the compatibility of blow-ups and pullbacks, one can reduce to the case when $S$ is affine, and so we can take $E$ to be free and $E_1 = \SO$. We need to check that $B: = Bl_\pm(n,E)$ has the universal property of the blow-up, \textit{i.e.}\ it is final in the category of schemes over $X: = OG_\pm(n,E)$ for which the preimage of $Z:=OG^{E_1}_\pm(n,E)$ is an effective Cartier divisor. Let's check that $B$ is an object of this category. There is an exact sequence 
$$\xymatrix{0\ar[r] & E_1 \ar[r] & \tLn \ar[r] & \tLn/E_1 \ar[r] & 0}$$
of vector bundles over $Y : = OG^{\EC_1}_\mp(n,E) $.\ By dualizing it, we note that $B \cong \mathbb{P}_Y((\tLn)^\vee)$ (resp.\ $E(n,E) \cong \mathbb{P}_Y((\tLn/E_1)^\vee)$) is a projective bundle of relative dimension $n-1$ (resp.\ $n-2$) over $Y$, %cf. \cite[Lemma 1.11]{BC12}. 
and $E(n,E)$ is an effective Cartier divisor of $B$. 

Suppose that $f: W \to X$ is a morphism for which $f^{-1}Z$ is an effective Cartier divisor. On $X$ consider the morphism of bundles $s: \Ln \to E_1^\vee\cong \SO_X$  given by the composition $\Ln \to E \to E_1^\vee \cong E/E_1^\perp$. Recall that $Z$ is the zero locus of this morphism, \textit{i.e.} $Z$ is defined by the ideal $\textnormal{im}(s) \subset \SO_X$. By assumption, $\mathcal{I}: = f^{-1} \textnormal{im}(s)\subset \SO_W$ is invertible, and $L_{n,W}^\pm \to \mathcal{I}$ is surjective. Let $K:= \ker(L_{n,W}^\pm \to \mathcal{I})$.

Now note that $K \subset E_1^\perp$ and $ K $ is a bundle of dimension $n-1$ inside the Lagrangian $L_{n,W}^\pm$ of $E_W$. Therefore, there is a unique Lagrangian $\tilde{L}_{n,W}^\mp$ of $E_W$ such that $\tilde{L}_{n,W}^\mp$ is in the opposite component from $L_{n,W}^\pm$ and that $K \subset \tilde{L}_{n,W}^\mp$. We claim that $E_1 \subset  \tilde{L}_{n,W}^\mp$.\ Note that $(K+E_1)/K$ is a Lagrangian of $E_1^\perp/E_1$. This is because 
$$(K+E_1)/K \subset (K^\perp\cap E_1^\perp)/K = ((K+E_1)/K)^\perp$$
 (cf.\ \cite[Section 6]{QSS}), and $(K+E_1)/K$ is a rank one bundle.\ If not $K+E_1 = K$, then $E_1 \subset K \subset L_{n,W}^\pm$ and $L_{n,W}^\pm \subset E_1^\perp$, which contradicts the subjectivity of $L_{n,W}^\pm \to \mathcal{I} \subset E/E_1^\perp$ and $\mathcal{I}$ is invertible.) Since $K^\perp /K$ is of rank two, we see that $(K+E_1)/K$ is either $ \tilde{L}_{n,W}^\mp/K$ or $L_{n,W}^\pm/K$. If $(K+E_1)/K = L_{n,W}^\pm/K$, then $E_1 \subset  L_{n,W}^\pm$ which is a contradiction, as we have already seen. Thus, $E_1 \subset \tilde{L}_{n,W}^\mp$, which shows that $W$ is a scheme over $Y$.\ By the universal property of the Grassmannian bundle $B: = Gr_Y(n-1,\tLn)$, the inclusion $K \hookrightarrow \tilde{L}_{n,W}^\mp$ defines the wanted morphism $g: W \to B$ such that $\pi\circ g = f$.

For the uniqueness, if $g': W \to B $ is a map such that $\pi \circ g' = f$, then $g'^*(\Ln) = L_{n,W}^\pm$ and $g'^*P_{n-1}^\mp \subset  L_{n,W}^\pm$. Note that $g'^*P_{n-1}^\mp\subset E_1^\perp$ on $W$, which forces $g'^*P_{n-1}^\mp\subset K$ and hence to be equal by dimension counting. Thus, $g'^*\tilde{L}_{n}^\mp = \tilde{L}_{n,W}^\mp$ is the unique Lagrangian in the other component containing $K = g'^*P_{n-1}$, which implies that $g = g'$. 

Finally, we see that $U_\pm = B - E(n,E) \cong  \mathbb{P}_Y((\tLn)^\vee) - \mathbb{P}_Y((\tLn/E_1)^\vee) \cong X-Z$. The scheme $ \mathbb{P}_Y((\tLn)^\vee) - \mathbb{P}_Y((\tLn/E_1)^\vee)$ is an affine bundle over $Y$, but need not to be a vector bundle in general.   
\end{proof}

Consider the case of $OG_\pm(2,E)$. Suppose that $(E,\beta)$ admits a complete flag
$$0 \subset E_1 \subset E_2 = E_2^\perp \subset E_1^\perp \subset E.$$
By Remark \ref{rmk:tildeE}, there is a Lagrangian $\tilde{E}_2$ in the other component such that the filtration
  $$0 \subset E_1 \subset \tilde{E}_2 = \tilde{E}_2^\perp \subset E_1^\perp \subset E$$
  is a complete flag. 
\begin{Prop}\label{Prop: OG2}
There are isomorphisms of schemes $OG_+(2,E) \cong \mathbb{P}(\tilde{E}_2^\vee)$ and $OG_-(2,E) \cong \mathbb{P}(E_2^\vee)$.  
\end{Prop}
%In particular, this shows that $OG_+(2,E)$ can not be isomorphic to $OG_-(2,E)$ in general. Although two connected components are isomorphic in the split case. 
\begin{proof} Recall that $\iota_\pm:OG_\pm^{E_1}(2,E)  \cong S$.\ Note that the pullback $\tilde{L}_2^+$ of the universal bundle $L_2^+$ on $OG_+(2,E)$ via the morphism $\iota_+:OG_+^{E_1}(2,E) \hookrightarrow OG_+(2,E) $ is isomorphic to $E_2$, and $\tilde{L}_2^-$ is isomorphic to $\tilde{E}_2$. Note that $OG_+^{E_1}(2,E)$ is a  closed subscheme of $OG_+(2,E)$ of codimension one. Therefore, the morphism $\pi_+: Bl_+(2,E) \cong \mathbb{P}(\tilde{E}_2^\vee) \to OG_+(2,E)$ is an isomorphism of schemes, and similarly $OG_-(2,E) \cong \mathbb{P}(E_2^\vee)$. 
\end{proof}
\subsection{Picard groups} In this subsection, we study Picard groups of spinor varieties. 
\begin{Prop}[$\textnormal{Theorem 1.3 \cite{BC}}$]\label{prop:K-CH} Let $H^*$ be a homotopy invariant cohomology theory for regular schemes.\ Assume that $H^*$ is oriented, so that it admits push-forward maps satisfying flat-base change, which are defined for all proper morphisms (examples of such theories are $K$-theory and Chow groups).\ Then, we have the following equivalence
	\[ H^*(OG_\pm(n,E)) \cong
	\bigoplus\limits_{i = 1}^{2^{n-1}} H^*(S)   . \]
\end{Prop}
\begin{proof}
	First of all let us observe that if $\rk(\EC^{n-1}) =2$, we can make the identification $S = OG_\pm(1,\EC^{n-1})$. Finally, we inductively apply Lemma \ref{lem:embeddingPhi}, homotopy invariance and \cite[Theorem 1.3]{BC}.
\end{proof}

\begin{Prop}\label{lem:lambda}
	$\Pic(OG_\pm(n,E)) \cong
	\mathbb{Z} \oplus \Pic(S)  $
	and $\Pic(Bl_\pm(n,E)) \cong \mathbb{Z}[\SO(E(n,E))] \oplus \Pic(OG_\pm(n,E)).$
\end{Prop}
\begin{proof}
	Let us recall that for $X$ regular, one can identify $\Pic(X) = CH^1(X)$. The first formula is then obtained by Proposition \ref{prop:K-CH}. Note that we have the following isomorphisms induced by pullbacks
	$$CH^1(OG_+(n,E))\cong CH^1(OG_+^{E_1}(n,E)) \cong \cdots \cong CH^1(OG_+^{E_{n-2}}(n,E)) \cong CH^1(OG_+(2, E^{n-2})) $$
The last group is isomorphic to $CH^1(\mathbb{P}(\tilde{E}^{n-2}_2)^\vee)$, which is itself equal to $\mathbb{Z} \oplus \Pic(S) $. The second formula is obtained by \cite[Proposition A.6 (i)]{BC}. A similar computation applies to $OG_-(n,E)$.  
\end{proof}
\begin{remark}
Note that in the isomorphism $\Pic(OG_+(n,E)) \cong \mathbb{Z} \oplus \Pic(S)  $, the group $\mathbb{Z}$ is generated by $\SO(1)$ coming from $OG_+(2,E^{n-2}) \cong \mathbb{P}((\tilde{E}^{n-2}_2)^\vee)$ which is known as the \textit{Pfaffian line bundle}, and $\Pic(S)$ comes from pullback from $S$. 
\end{remark}
\subsection{The canonical bundle}
\begin{lem}\label{lem:determinanttautologicalbundle}
$\det L_n^+ \cong \det \tilde{E}_n \otimes \SO(-2)$. 
\end{lem}
\begin{proof}
Consider the case when $n=2$. By Lemma \ref{lem:tangentOGn}, we see that $\det L_2^+ \cong T_{OG_+(2,E)/S}^\vee$, where the last bundle is isomorphic to $T_{\mathbb{P}(\tilde{E}_2^\vee)/S}^\vee$. Now, the exact sequence (cf. \cite[Appendix B5.8]{Fu})
$$\xymatrix{0 \ar[r] & \SO_{\mathbb{P}(\tilde{E}_2^\vee)/S} \ar[r] & \tilde{E}_2^\vee \otimes \SO(1) \ar[r] & T_{\mathbb{P}(\tilde{E}_2^\vee)/S} \ar[r] & 0} $$ 
shows that $\det T_{\mathbb{P}(\tilde{E}_2^\vee)/S}^\vee \cong \det ( \tilde{E}_2^\vee \otimes \SO(1) )^\vee \cong \SO(-2) \otimes  \det  \tilde{E}_2    $. 

 The general case can be reduced to  $n=2$. There are exact sequences
$0\to E_{n-2} \to L_n^+ \to L_n^+/E_{n-2} \to 0$ and $0\to E_{n-2} \to \tilde{E}_n \to \tilde{E}_n/E_{n-2} \to 0$ on $OG_+^{E_{n-2}}(n,E)$. Here $L_n^+/E_{n-2}$ can be identified with the universal bundle on $OG_+(2, E^{n-2})$.
%Next, we claim that $ \det  \tilde{E}_2^\vee \cong \det E_2$. There are two exact sequences $\xymatrix{0 \to E_1 \to E_2 \to E_2/E_1 \to 0} $ and $\xymatrix{0 \to E_1 \to \tilde{E}_2 \to \tilde{E}_2/E_1 \to 0} $, and $\tilde{E}_2/E_1 \cong (E_2/E_1)^\vee$ by Proposition \ref{Prop: rank two}. Hence, $\det E_2 \cong E_1 \otimes (E_2/E_1)$ and $\det \tilde{E}_2 \cong E_1 \otimes (E_2/E_1)^\vee $. It follows that $ \det \tilde{E}_2^\vee \cong \det E_2 \otimes (E_1^\vee)^{\otimes 2} $. 
Thus, we conclude that $\det L_n^+ \cong \det \tilde{E}_n \otimes \SO(-2)$.
\end{proof}

Let $N_{\iota_\pm}$ be the normal bundle of the closed embedding $\iota_\pm: OG_\pm^{\EC_1}(n,E) \hookrightarrow OG_\pm(n,E)$. We want to compute the class $\omega_{\iota_\pm}: = \det N_{\iota_\pm}^\vee$ in $\Pic(OG_\pm^{E_1}(n,E))$ as it is important for Witt groups.
\begin{Prop}\label{prop:canonicalbundle}
$\omega_{\iota_+} \cong \SO(-2)\otimes\det \tilde{E}_n \otimes E_1^{\otimes (n-2)}$.
\end{Prop}
\begin{proof}
On the one hand, the exact sequence
$$\xymatrix{0 \ar[r] & T_{OG^{E_1}_+(n,E)/S} \ar[r] &  \iota_+^* T_{OG_+(n,E)/S} \ar[r] & N_{\iota_+} \ar[r] & 0 }$$
shows that
$$ \det  N_{\iota_+} \cong  \det \iota_+^* T_{OG_+(n,E)/S} \otimes \det  T_{OG^{E_1}_+(n,E)/S}^\vee\,. $$
Let $U$ be the cokernel of the embedding $ E_1 \to L_n^+ $ on $OG^{E_1}_+(n,E)$. Since $OG^{E_1}_+(n,E) \cong OG_+(n-1,E^1)$, the bundle $U$ can be identified with the universal bundle on $OG_+(n-1,E^1)$.\ The tangent bundle of $OG^{E_1}_+(n,E)$ is isomorphic to $\Lambda^2 U^\vee$. Therefore, $\det  N_{\iota_+} \cong \det \Lambda^2 (L_n^+)^\vee \otimes \det  \Lambda^2 U $. On the other hand, note that the exact sequence $0 \to U^\vee \to (L_n^+)^\vee \to E_1^\vee \to 0 $ on $OG^{E_1}_+(n,E)$ induces the following exact sequence 
$$\xymatrix{ 0 \ar[r] &\Lambda^2 U^\vee \ar[r] & \Lambda^2 (L_n^+)^\vee \ar[r] & U^\vee \otimes E_1^\vee \ar[r] & 0\,, }$$
which shows that
$$\det  N_{\iota_+} \cong \det \Lambda^2 (L_n^+)^\vee \otimes \det  \Lambda^2 U \cong \det (U^\vee \otimes E_1^\vee) \cong \det U^\vee \otimes (E_1^\vee)^{\otimes(n-1)}  \cong \det (L_n^+)^\vee \otimes (E_1^\vee)^{\otimes (n-2)} . $$
Now,  by Lemma \ref{lem:determinanttautologicalbundle}, we see that $$\det  N_{\iota_+}^\vee \cong \det (L_n^+) \otimes (E_1^\vee)^{\otimes (n-2)} \cong \SO(-2)\otimes\det \tilde{E}_n \otimes E_1^{\otimes (n-2)}.\qedhere$$
\end{proof}
\vspace{0.2 cm}
\begin{remark} A similar argument shows that
 $\omega_{\iota_-} \cong \SO(-2)\otimes\det E_n \otimes E_1^{\otimes (n-2)}.$
	%Assume that $\EC$ is split hyperbolic.
	%Note that the closed embedding $\iota: OG_+^{\EC_1}(n,E) \hookrightarrow OG_+(n,E)$ is a regular embedding of codimension of $(n-1)$.
	%Define $\omega_{\iota}$ to be the dual of the determinant of the normal bundle of $\iota$. This is not isomorphic to the generator $\mathcal{O}(1)$ in the Picard group of $OG_+^{\EC_1}(n,\EC)$, instead $\omega_{\iota}$ is trivial in $\Pic(OG_+^{\EC_1}(n,\EC))/2$. 
	%Recall the number $\lambda(\LC)$ in \cite{BC}.  Note that the number  $\lambda(\mathcal{O}) = 0 \in \Pic(Bl)$.
\end{remark}

\section{Witt groups of maximal isotropic Grassmannian}
\subsection{Some preliminary results}
The following theorems are useful for our computations.
\begin{theo}[Localization]\label{thm: Balmer}
	Let $X$ be a scheme with $\frac{1}{2} \in \SO_X$. Let $Z$ be a closed subset and $U$ be its open complement. Let $\LC$ be a line bundle on $X$. Then, there is a $12$-term periodic long exact sequence
	\[ \cdots \to W^{i-1}(U,\LC_U) \xrightarrow{\partial} W^i_Z(X,\LC) \to W^i(X,\LC) \to W^i(U,\LC_U) \xrightarrow{\partial} W^{i+1}_Z(X,\LC) \to \cdots \,. \]
\end{theo}
For the proof we refer the reader to Balmer \cite[Theorem 1.5.5]{Balmer}.
\begin{theo}[D\'evissage]
	Let $\iota: Z \hookrightarrow X$ be a regular embedding of codimension $d$, where $Z$ and $X$ are both regular separated schemes of finite Krull dimension with $\frac{1}{2}$ in their global sections. Let $\omega_\iota$ be dual of the determinant of the normal bundle $N_\iota$ of the embedding $\iota$. Then, there is an isomorphism
	\[ \iota_* : W^{i-d}(Z, \omega_\iota \otimes \LC_Z) \to W^i_Z(X,\LC). \]
\end{theo}
The above theorem is proved by Gille in \cite{Gille} (see also \cite[Corollary 6.2]{Xie20}).\ Combining the two theorems together, one obtains the following corollary.
\begin{Coro}
	Let $\iota: Z \hookrightarrow X$ be a regular embedding of codimension $d$, where $Z$ and $X$ are both regular separated schemes of finite Krull dimension with $\frac{1}{2}$ in their global sections. Let $\omega_\iota$ be dual of the determinant of the normal bundle $N_\iota$ of the embedding $\iota$. Let $U: = X -Z$. Then, one has the following $12$-term periodic long exact sequence
	\[ \cdots \to W^{i-1}(U,\LC_U) \xrightarrow{\partial} W^{i-d}(Z,\omega_\iota \otimes \LC_Z) \xrightarrow{\iota_*} W^i(X,\LC) \to W^i(U,\LC_U) \xrightarrow{\partial} W^{i-d+1}(Z,\omega_\iota \otimes \LC_Z) \to \cdots . \]
\end{Coro}

\begin{theo}[Homotopy invariance]
	Let $X$ be a regular scheme with $\frac{1}{2} \in \SO_X$.\ Assume that $p:E \to X$ is an affine bundle and let $\LC$ be a line bundle on $X$. Then the pullback induces an isomorphism
	\[ p^* :  W^i(X,\LC) \xrightarrow{\ \cong\ } W^i(E, p^*\LC). \]
\end{theo}
A proof can be found either in \cite{Balmer 01} or in \cite{Gille03}. For the singular case see instead \cite{KSW20}.
\begin{lem}[Two periodicity on the twists]\label{lem:twoperio}
	Let $\mathcal{L}$ and $\mathcal{M}$ be line bundles on $X$. Then we have an isomorphism 
	$$\mathrm{per}_{\mathcal{M}}: W^i(X, \mathcal{L}) \rightarrow W^i(X, \mathcal{L} \otimes \mathcal{M}^{\otimes 2}),$$
	which, by definition, is the product $(\mathcal{M},\varphi) \otimes - $, where  $\varphi \colon\mathcal{M} \rightarrow \mathcal{H}om_{\SO}(\mathcal{M}, \mathcal{M}^{\otimes 2} )$ is the canonical form in $ W^0(X, \mathcal{M}^{\otimes2} )$.
\end{lem}
See \cite{Balmer} for a proof.

\begin{theo}[Projective bundles]\label{thm:Walter-Nenashev} Let $E$ be a vector bundle over $S$ with odd rank $r$ such that $E$ fits into an exact sequence
$ {0 \to  E' \to  E \to  E'' \to  0. }$ Let $q: \mathbb{P}(E) \to S$ be the projection morphism.\ Suppose that $\rk(E')-1 \equiv \rk(E'')-1 \equiv 0 \mod 2$.\ Then, for any line bundle $L$ over $S$, there are isomorphisms
	$$W^i(\PS(E),q^*L) \cong W^{i-r}(S,L \otimes \det E) \oplus W^{i}(S,L) \quad \textnormal{  and  } \quad  W^i(\PS(E) ,q^*L \otimes \SO(1)) = 0$$
	for any $i \in \mathbb{Z}$.  
\end{theo}
See Nenashev \cite{Ne} and Walter\cite{Walter}. 
\begin{remark}
	It is worth pointing out that, since for regular schemes Witt groups on vector bundles and those on coherent sheaves are isomorphic, in this paper we do not strictly distinguish between the two. 
\end{remark}

\subsection{Two simple cases} 
Let $\LC$ be a line bundle on $OG_+(n,E)$. By homotopy invariance,  d\'evissage theorem and Proposition  \ref{prop:canonicalbundle}, the 12-term periodic long exact sequence takes the form
\begin{equation}\label{eq:long-witt} \xymatrix{     W^{i-n+1}(OG^{E_1}_+(n,E), \LC \otimes \det \tilde{E}_n \otimes E_1^{\otimes (n-2)})  \ar[r]^-{(\iota_{+})_*} &            W^i(OG_+(n,E), \LC) \ar[r]^-{v_+^*} & W^i(U_+(n,E), \LC) \ar[dll]_{\partial}\\
		  W^{i-n+2}(OG^{E_1}_+(n,E), \LC \otimes \det \tilde{E}_n \otimes E_1^{\otimes (n-2)}) \ar[r] & \cdots   &\phantom{a} \hspace{2.6 cm}.}\end{equation}
Here, we suppressed even twists and we made a choice of a trivialization. A similar exact sequence can be written on $OG_-(n,E)$ replacing $+$ with $-$ and $\tilde{E}_n$ with $E_n$. The following result is proved in \cite{Calmes-Fasel} when the base is a field of characteristic $\neq 2$. 
\begin{Prop}\label{thm:L2d-1} 
	For any line bundle $L$ over $S$, one has $W^i(OG_\pm(n,E), L \otimes \mathcal{O}(1)) = 0$. 
\end{Prop}
\begin{proof}
Recall that both $OG_+(2,E)$ and $OG_-(2,E)$ are projective bundles satisfying the condition of Theorem \ref{thm:Walter-Nenashev}. It follows that $W^i(OG_\pm(2,E), L \otimes \SO(1)) = 0$ for any line bundle $L$ over $S$ and any $i \in \mathbb{Z}$.\ By inductively applying the localization sequence (\ref{eq:long-witt}), the result follows from Theorem \ref{theo:blowup}, homotopy invariance and Lemma \ref{lem:embeddingPhi}. 
\end{proof}

\begin{Prop}\label{thm:evendim-eventwist}
	Let $L$ be a line bundle over $S$ and assume  $n$ to be even. Then, the sequence
	\[ \xymatrix{   0 \ar[r] &    W^{i-(n-1)}(OG^{E_1}_+(n,E), L \otimes \det \tilde{E}_n)  \ar[r]^-{(\iota_{+})_*} &    W^i(OG_+(n,E),L) \ar[r]^-{\alpha^*_+} & W^i(U_+(n,E),L ) \ar[r] &0 }\]
	is split exact. Hence, we have the following isomorphism
	\[ W^i(OG_+(n,E), L) \cong W^{i-(n-1)}(OG^{E_1}_+(n,E),L \otimes \det \tilde{E}_n)  \oplus W^i(OG_-^{E_1}(n,E),L). \]
\end{Prop}
\begin{proof}
	As $n$ is even, \cite[Theorem 1.4 (A)]{BC} applies here and, as a consequence, the localization sequence breaks down into short split exact sequences. Finally, the last statement follows directly from homotopy invariance.
\end{proof}

\subsection{Preparation for the case of $n$ odd}
% Let $n$ be an odd number.
It now remains to compute $W^i(OG_\pm(n,E),L)$ for $n$ odd. This case is more complicated and forms the technical heart of the paper.\	We only deal with one component, the other one can be obtained \textit{mutatis mutandis}.  Let us begin by observing that the proof of Proposition 3.3 in \cite{BC} provides more information, which will be useful for the proof. Rewrite Diagram (\ref{eq:keydiagram}) starting with the canonical closed embedding $\iota_{1+}: OG^{\EC_2}_+(n,E)   \to OG^{E_1}_+(n,E)$ as follows:
	\begin{equation}\label{eq:keydiagram2}
		\xymatrix{OG^{\EC_2}_+(n,E)  \ar[rr]^-{\iota_{1+}} && OG^{E_1}_+(n,E)   && \ar[ll]_-{v_{1+}} U^{E_1}_{+}(n,E) \ar[d]^-{\alpha_{1-}} \ar[lld]_-{\tilde{v}_{1+}} \\
		%%%%%%%%%%%%%%%%
			E^{E_1}(n,E)  \ar[rr]^-{\tilde{\iota}_{1+}} \ar[u]^-{\tilde{\pi}_{1+}} \ar@/_3pc/[rrrr]_-{\tilde{\pi}_{1-}} && Bl^{E_1}_+(n,E) \ar[u]^-{\pi_{1+}} \ar[rr]^-{\tilde{\alpha}_{1-}} && OG^{\EC_2}_-(n,E) }
	\end{equation} 
At the level of functor of points the previous diagram can be described by
	\[  \xymatrixcolsep{1pc}
	\begin{footnotesize}
	\xymatrix{
		 \left\{L_n^+ \ \middle|\: \  L_n^+ \supset E_2\right\}
		\ar[r] 
		&\left\{L_n^+ \ \middle|\: \ L_n^+ \supset E_1 \right\}
		%\ar[r] 
		&& 	\left\{ L_n^+\ \middle|\: \
		{\begin{matrix} 
				L_n^+ \not\supset E_2
			\\ 	L_n^+ \supset  E_1 	\end{matrix}}   \right\} 
		\ar[dll]\ar[d]\ar[ll] 		\\
		% % % % % % % % % % % % % % % % % % % % % % % % % % % % % % % % % % % % % % % % % % % % % % % % % % % % % % % % % %
		\left\{(P_{n-1},L_n^+,L_n^-)\ \middle|\: \
		{\begin{matrix} 
	      L_n^+\supset P_{n-1} 
		\\ 	L_n^- \supset P_{n-1}  
		\\ P_{n-1} \supset E_2
			\end{matrix}}   \right\} 
		\ar[r]\ar[u]  
		&\left\{(P_{n-1},L_n^+, L_n^- )\ \middle|\: \ 
		{\begin{matrix}  
				L_n^+\supset E_1 
				\\ L_n^-\supset E_2 
			    \\L_n^+ \supset P_{n-1} 
			    \\L_n^- \supset P_{n-1}
			\end{matrix}} \right\}
		\ar[rr]\ar[u]
		&& \left\{ L_n^- \ \middle|\: \
		{\begin{matrix}L_n^- \supset E_2 
		 	\end{matrix}}  \right\}  	}
	\end{footnotesize} 
	\] 
Note that, in Diagram (\ref{eq:keydiagram2}), $\tilde{\iota}_{1+}$ is a regular embedding of codimension one, while $\iota_{1+}$ is a regular embedding of codimension $n-2$.  Suppose that $n$ is odd.\ As in Proposition \ref{thm:evendim-eventwist} and \cite[Theorem 1.4 (A)]{BC}, the short exact sequence in the following diagram
\[\xymatrix{   0 \ar[r] &    W^{i-(n-2)}(OG^{E_2}_+(n,E), \omega_{\iota_{1+}})  \ar[r]^-{(\iota_{1+})_*} &            W^i(OG^{E_1}_+(n,E)) \ar[r]^{v_{1+}^*} & W^i(U^{E_1}_+(n,E)) \ar[r] &0  
	\\%
	& & W^i(Bl_+^{E_1}(n,E)) \ar[u]^-{(\pi_{1+})_*} & W^i(OG_-^{E_2}(n,E)) \ar[l]^{\tilde{\alpha}_{1- }^*} \ar[u]^-{\alpha_{1- }^*}_-{\cong}
}   \]
has a right splitting:
$$\beta : =   (\pi_{1+})_* \circ \tilde{\alpha}_{1- }^* \circ (\alpha_{1- }^*)^{-1}.$$  

\begin{remark}\label{rmk:canonicalbundleinn-1}
	By Proposition \ref{prop:canonicalbundle} and Lemma \ref{lem:embeddingPhi}, we see that $\omega_{\iota_{1+}} \cong \SO(-2) \otimes \det \tilde{E}_n \otimes E_1^{\otimes (n-4)}$.\ 
\end{remark}

\begin{lem}\label{lem:pi1inverse} Assume that $n$ is an odd number. The maps $$(\pi_{1+})_*: W^{i}(Bl_+^{E_1}(n,E)) \leftrightarrows W^{i}(OG^{E_1}_+(n,E)) : \pi_{1+}^* $$
	are isomorphisms and inverse to each other.   \end{lem}
\begin{proof}
	The excess normal bundle of the left cartesian square in Diagram (\ref{eq:keydiagram2}) is precisely the universal quotient bundle $\mathcal{Q}$ of the projective bundle $\pi_{1+}: E^{E_1}(n,E) \rightarrow OG^{E_2}_+(n,E)$, which is the cokernel of the map $\SO_{E^{E_1}(n,E)}(-1) \rightarrow \tilde{\pi}_{1+}^*\mathcal{N}_{\iota_{1+}}$. 

By Lemma \ref{lem:projblow}, it is enough to prove that $ \pi_{1+}^*$ is an isomorphism.\ Consider the following diagram
	\[ \xymatrix{ 0 \ar[r] & W^{i-1}(E^{E_1}(n,E), \SO(-1)) \ar[r]^-{(\tilde{\iota}_{1+})_*} & W^i(Bl_+^{E_{1}}(n,E)) \ar[r]^-{\tilde{v}_{1+}^*} & W^i(U_+^{E_{1}}(n,E)) \ar[r] & 0  
		\\%%%%%%%%%%%%%%
		0 \ar[r] & W^{i-(n-2)}(OG^{E_2}_+(n,E), \det \tilde{E}_n\otimes E_1 ) \ar[u]^-{e(\mathcal{Q}) \circ \tilde{\pi}_{1+}^*} \ar[r]^-{(\iota_{1+})_*} & W^i(OG^{E_{1}}_+(n,E)) \ar[u]^-{\pi_{1+}^*} \ar[r]^-{v_{1+}^*} & W^i(U_+^{E_1}(n,E))\ar[u]^-{=}  \ar[r] & 0 \,, }\]
	where commutativity of the left square is obtained by using the excess intersection formula (cf.\ \cite{Fasel09}) applied to the left cartesian square of Diagram (\ref{eq:keydiagram2}). 	
	 The left vertical %horizontal  
	arrow of the previous diagram is an isomorphism by Theorem \ref{thm:projectivebundleevendimoddtwist}, because $\tilde{\pi}_{1+}$ is a projective bundle of relative dimension $n-3$, which is even by hypothesis. 
\end{proof}

\begin{lem}\label{lem:E2split}
	Assume $n$ to be an odd number. The short exact sequence
	\[\xymatrix{   0 \ar[r] &    W^{i-(n-2)}(OG^{E_2}_+(n,E), \det \tilde{E}_n \otimes E_1)  \ar[r]^-{(\iota_{1+})_*} &   W^i(OG^{E_1}_+(n,E))   \ar[r]^{v_{1+}^*} & W^i(U^{E_1}_+(n,E)) \ar[r] &0 } \]
	has a left splitting
	$   (\tilde{\pi}_{1+})_* \circ (\tilde{\pi}_{1-})_*^{-1}  \circ (\tilde{\alpha}_{1- })_* \circ  \pi_{1+}^*.$ 
\end{lem}
\begin{proof} By the projective bundle formula (cf.\ Theorem \ref{thm:projectivebundleevendimoddtwist}), we see that  $(\tilde{\pi}_{1-})_* $ and $ (\tilde{\pi}_{1+})_*$ are isomorphisms. Consider the following diagram
	\[ \xymatrix{
		W^{i-1}(E^{E_1}(n,E), \SO(-1)) \ar[r]^-{(\tilde{\pi}_{1-})_*}_-{\cong}   \ar@/_8pc/[dd]_-{(\tilde{\pi}_{1+})_*}^-{\cong}  \ar@{=}[d]  &   W^{i-(n-2)}(OG_-^{E_2}(n,E),  \det E_n \otimes E_1) \\
		W^{i-1}(E^{E_1}(n,E), \SO(-1)) \ar[r]^-{(\tilde{\iota}_{1+})_*}   & W^i(Bl_+^{E_1}(n,E)) \ar[u]^-{(\tilde{\alpha}_{1- })_*}
		\\ %
		W^{i-(n-2)}(OG^{E_2}_+(n,E), \det \tilde{E}_n \otimes E_1 ) \ar[u]^-{e(\mathcal{Q}) \circ \tilde{\pi}_{1+}^*}_-{\cong} \ar[r]^-{(\iota_{1+})_*} & W^i(OG^{E_1}_+(n,E)) \ar[u]^-{\pi_{1+}^*}_-{\cong} .
	}\]
	The result follows by its commutativity.
\end{proof}

\begin{lem}\label{lem:OGE2splits} The sequence 
	\[\xymatrix{   0 \ar[r] &    W^{i-(n-2)}(OG^{E_2}_+(n,E), \det \tilde{E}_n \otimes E_1)  \ar[r]^-{(\iota_{1+})_*} &            W^i(OG^{E_1}_+(n,E)) \ar[r]^{\iota_{1+}^*} & W^i(OG^{E_2}_+(n,E)) \ar[r] &0  	}   \]
is split exact.  Moreover, $ W^i(OG^{E_2}_+(n,E)) \cong  W^i(OG^{E_2}_-(n,E))$. 
\end{lem}
\begin{proof}
	Consider the following diagram
	\[ \xymatrix{ 0 \ar[r] & W^{i-1}(E^{E_1}(n,E), \SO(-1)) \ar[r]^-{(\tilde{\iota}_{1+})_*} & W^i(Bl_+^{E_{1}}(n,E)) \ar[r]^-{\tilde{\iota}_{1+}^*} & W^{i}(E^{E_1}(n,E)) \ar[r] & 0  
		\\%%%%%%%%%%%%%%
		0 \ar[r] & W^{i-(n-2)}(OG^{E_2}_+(n,E), \det \tilde{E}_n\otimes E_1 ) \ar[u]^-{e(\mathcal{Q}) \circ \tilde{\pi}_{1+}^*}_-{\cong} \ar[r]^-{(\iota_{1+})_*} & W^i(OG^{E_{1}}_+(n,E)) \ar[u]^-{\pi_{1+}^*}_-{\cong} \ar[r]^-{\iota_{1+}^*} & W^i(OG_+^{E_2}(n,E))\ar[u]^-{\tilde{\pi}_{1+}^*}_-{\cong}  \ar[r] & 0 \ . }\]
All vertical maps are isomorphisms. The commutativity of the left square diagram (resp. the right square) follows from the excess intersection formula (resp.\ functoriality of pullback). The upper sequence is split exact by Lemma \ref{lem:chowmoving}. Note that $Bl_+^{E_{1}}(n,E)$ (resp. $E^{E_1}(n,E)$) is a projective bundle over $OG_+^{E_2}(n,E)$ of relative dimension $n-2$ (resp. $n-3$), and the condition of Lemma \ref{lem:chowmoving} is satisfied (cf. Proof of Theorem \ref{theo:blowup}). This proves the first statement. The last statement follows from the first one in combination with homotopy invariance and Lemma \ref{lem:E2split}. 
\end{proof}

\subsection{The main theorem}
We can now state our main result.
\begin{theo}\label{thm:odd dimension case}
There is an isomorphism of graded $W^\tot(S)$-modules
	\begin{equation}\label{eq:nodd}
		W^\tot(OG_+(n,E)) \cong W^{\tot}(OG^{E_2}_+(n,E))\big[-(2n-3)\big] \oplus W^\tot(OG^{E_2}_+(n,E))
	\end{equation}
	if $n \geq 3$ is odd.
\end{theo}

\begin{proof} By \cite[Theorem 1.4 (B)]{BC}, the connecting homomorphism $\partial$ in the localization sequence (\ref{eq:long-witt}) can be fitted into the commutative diagram
\[ \xymatrix{ W^i(U_+(n,E) )  \ar[r]^-{\partial} &  W^{i-(n-2)}(OG_+^{E_1}(n,E),  \det{\tilde{E}_n}) \\
	W^{i}(OG^{E_1}_-(n,E)) \ar[u]^-{\alpha_+^*}_-{\cong} \ar[r]^{\tilde{\pi}_-^*} & W^{i}(E(n,E)) \ar[u]^-{(\tilde{\pi}_+)_*}  .}\]
From now on, for simplicity, we shall suppress all the twists, since they can be recovered if necessary. Our strategy is to further decompose the $n-1$ cases (\textit{i.e.} when one has $OG_\pm^{E_1}(n,E)$) into the $n-2$ cases and analyze the map $(\tilde{\pi}_+)_*  \circ \tilde{\pi}_-^*$.\ Look at the following picture
\[ \small \xymatrix{
	W^{i}(OG^{E_2}_+(n,E)) \ar[dr]|-{\ \beta\ } \ar[d]^-{\tilde{\alpha}_{1+ }^*} \ar[r]^-{\alpha_{1+}^*}_-{\cong} &   W^i(U^{E_1}_-(n,E))    &  W^{i-2(n-2)}(OG^{E_2}_+(n,E)) \ar[d]^-{(\iota_{1+})_*}
	&\ar[l]_-{\epsilon}^-{\cong} W^{i-2(n-2)}(OG_-^{E_2}(n,E)) &
	\\ %
	W^i(Bl_-^{E_1}(n,E)) \ar[r]^-{(\pi_{1-})_*} & W^i(OG^{E_1}_-(n,E)) \ar[r]^-{(\tilde{\pi}_+)_*  \circ \tilde{\pi}_-^*}  \ar[u]^-{v_{1-}^*}& W^{i-(n-2)}(OG^{E_1}_+(n,E) ) \ar[d]^-{v_{1+}^*} \ar[r]^-{\pi_{1+}^*} \ar[ur]|-{\ \gamma\ } & W^{i-(n-2)}(Bl_+^{E_1}(n,E)) \ar[u]_-{(\tilde{\alpha}_{1- })_*}&
	\\ %
	&W^{i-(n-2)}(OG^{E_2}_-(n,E))  \ar[u]^{(\iota_{1-})_*}& W^{i-(n-2)}(U^{E_1}_+(n,E))& &\phantom{a}\hspace{-1.5 cm},}\]
where $\beta$ and $\gamma$ are defined as the depicted composition, and $ \epsilon : =  (\tilde{\pi}_{1+})_* \circ (\tilde{\pi}_{1-})_*^{-1}$.

\begin{sublemma}\label{sublem:composition-iso}
	The composition $v_{1+}^* \circ (\tilde{\pi}_+)_*  \circ \tilde{\pi}_-^* \circ (\iota_{1-})_*$ is an isomorphism.
\end{sublemma}
\begin{proof}[Proof of Sublemma \ref{sublem:composition-iso}]
	Consider the following diagram
	$$
	\xymatrix{
		U^{E_1}_+(n,E) \ar[r]^-{v_{1+}} \ar@{=}[dd] & OG^{\EC_1}_+(n,E)  \ar[r]^-{\iota_+} & OG_+(n,E)   & \ar[l]_-{v_+} U_+(n,E) \ar[d]^-{\alpha_+}  
		\\%
		& E(n,E)  \ar@/_1pc/[rr]_{\tilde{\pi}_-} \ar[r]^-{\tilde{\iota}_+} \ar[u]^-{\tilde{\pi}_+} & Bl_+(n,E) \ar[u]^-{\pi_+} \ar[r]^-{\tilde{\alpha}} & OG^{\EC_1}_-(n,E)
		\\%
		U^{E_1}_+(n,E)\ar[r]^-{\tilde{v}_{1+}} \ar@/_2pc/[rrr]_-{\alpha_{1- }} & Bl_+^{E_1}(n,E) \ar[rr] \ar[u] & & OG_-^{E_2}(n,E) \ar[u]_-{\iota_{1-}}\ ,  \\
		&
	}
	\\%
	$$
	where we make the following interpretation on functors of points:
	\[ \tiny \xymatrix{
		\left\{ L_n^+\  \middle|\: \ {\begin{matrix}
			L_n^+	 \supset E_1  \\ L_n^+ \not\supset E_2  
		\end{matrix}} \right\} \ar@{=}[dd]  \ar[r] & \left\{L_n^+ \ \middle| \: \  L_n^+ \supset  \EC_1 \right\} \ar[r] & \left\{L_n^+ \ \middle| \: \   E \supset  L_n^+  \right\}  & \ar[l]  \left\{L_n^+ \ \middle| \: \  L_n^+ \not\supset \EC_1 \ar[d] \right\}  
		\\%%%%%%%%%%%
		& \left\{(P_{n-1},L_n^+,L_n^-) \ \middle|\:  \
		{\begin{matrix} 
				L_n^+ \supset P_{n-1} \\
				L_n^- \supset P_{n-1} \\
				P_{n-1} \supset \EC_1  
		\end{matrix}} \right\}   \ar[u]  \ar[r] & \ar[u]
		\left\{(P_{n-1},L_n^+,L_n^-) \ \middle|\:  \
		{\begin{matrix}  L_n^-  \supset  \EC_1  \\ L_n^\pm \supset P_{n-1} 
		\end{matrix}} \right\}  \ar[r] & \left\{  L_n^-\  \middle|\: \ L_n^- \supset  \EC_1 \right\}
		\\%%%%%%%%%
		\left\{  L_n^+\  \middle|\: \ {\begin{matrix}
				 L_n^+ \supset E_1  \\ L_n^+ \not\supset E_2 
		\end{matrix}} \right\}  \ar[r] &   \left\{(P_{n-1}, L_n^+,L_n^-) \middle|\: \ 
		{\begin{matrix} L_n^+ \supset \EC_1  \\ L_n^- \supset E_2  \\  L_n^\pm \supset P_{n-1} 
		\end{matrix}} \right\}   \ar[u] \ar[rr]  & & \left\{  L_n^-\  \middle|\: \ L_n^- \supset E_2 \right\} \ . \ar[u] ,}\]  
	Note that the lower right rectangle diagram is cartesian and tor-independent.\ This follows from $\tilde{\pi}_-$ being flat. By the base change formula, we have
	$$v_{1+}^* \circ (\tilde{\pi}_+)_*  \circ \tilde{\pi}_-^* \circ (\iota_{1-})_*  = \alpha_{1- }^*$$
	and the claim follows, since $\alpha_{1- }$ is an affine bundle.
\end{proof}

\begin{sublemma}\label{sublem:composition-zero}
	The composition $(v_{1+})^* \circ  (\tilde{\pi}_+)_*  \circ \tilde{\pi}_-^* \circ \beta$ is zero.
\end{sublemma}
\begin{proof}[Proof of Sublemma \ref{sublem:composition-zero}]
	By Lemma \ref{lem:E2split} and \ref{lem:OGE2splits}, we see that $\ker(v_{1+}^*) = \ker (\iota_{1+}^*) = \textnormal{im} \, (\iota_{1+})_* $. Therefore, it is enough to show that $$ \iota_{1+}^* \circ (\tilde{\pi}_+)_*  \circ \tilde{\pi}_-^* \circ \beta = 0.$$ Note that we have the following diagram
	\begin{equation}\label{eq:caseII}
		\xymatrix{
			OG^{E_2}_+(n,E) \ar[r]^-{\iota_{1+}} & OG^{E_1}_+(n,E)   \\
			Bl_-^{E_1}(n,E) \ar@/_3pc/[rr]_{\pi_{1-}} \ar[r] \ar[u]^-{\tilde{\alpha}_{1+ }} &\ar[r] E(n,E)  \ar[r]^-{ \tilde{\pi}_-} \ar[u]^-{\tilde{\pi}_+} & OG^{E_1}_-(n,E) \ ,   }
	\end{equation}  
	where the left square is cartesian and tor-independent. % On functors of points, we have
%	\[ \scriptsize \xymatrix{ \left\{ L_n^+ \ \middle|\: \ L_n^+ \supset \EC_2 \right\} \ar[r] & \left\{  L_n^+ \ \middle|\: \  L_n^+ \supset \EC_1  \right\}     \\
%		\left\{(P_{n-1},L_n^+,L_n^-) \ \middle|\: \
%		{\begin{matrix} L_n^- \supset \EC_1 \\ L_n^+ \supset \EC_2 \\ L_n^\pm \supset P_{n-1} 
%		\end{matrix}} \right\}   \ar[u]  \ar[r] & \ar[u]
%		\left\{(P_{n-1},L_n^+,L_n^-) \ \middle|\: \
%		{\begin{matrix} L_n^\pm \supset \EC_1    \\  P_{n-1} \supset E_1
%		\end{matrix}} \right\}  \ar[r] & \left\{ L_n^- \ \middle|\: \ L_n^- \supset \EC_1  \right\} . }\]
	Observe that, in view of the functoriality of pullbacks and the base change formula \cite[Theorem 5.5]{CH11}, one has
	$$  \iota_{1+}^* \circ (\tilde{\pi}_+)_*  \circ \tilde{\pi}_-^* \circ \beta = (\widetilde{\alpha}_{1+})_* \circ  \pi_{1-}^* \circ  (\pi_{1-})_* \circ  \widetilde{\alpha}_{1+}^* \,. $$

	%(note that the square diagram is tor-independent since $\tilde{\pi}:E\rightarrow OG^{E_1}_+(n,E)$ is a projective bundle, hence flat).
	By Lemma \ref{lem:pi1inverse}, we see that the composition $\pi_{1-}^* \circ  (\pi_{1-})_*  $ is the identity map. Since $\widetilde{\alpha}_{1-}: Bl_-^{E_1}(n,E) \rightarrow OG^{E_2}_+(n,E)$ is a projective bundle of odd relative dimension,  Theorem \ref{thm:projectivebundleodddimeventwist} implies that $ (\widetilde{\alpha}_{1+})_*  \circ  \widetilde{\alpha}_{1+}^*  $ vanishes.
\end{proof}

\begin{sublemma}\label{sublem:composition-zero-1}
	The composition $ \gamma \circ (\tilde{\pi}_{+})* \circ \tilde{\pi}_-^* \circ (\iota_{1-})_*$ vanishes.
\end{sublemma}
\begin{proof}
	Consider the following diagram 
		$$
	\xymatrix{	   OG^{\EC_1}_+(n,E)   \ar@{=}[d]	&   \ar[l]_-{\tilde{\pi}_{1+}} E(n,E)  \ar[r]^-{\tilde{\pi}_-}  &  OG^{\EC_1}_-(n,E)
		\\%
 OG^{\EC_1}_+(n,E)  		 &  \ar[l]^-{\pi_{1+}}  Bl_+^{E_1}(n,E) \ar[r]^-{\tilde{\alpha}_{1-}} \ar[u]  & OG_-^{E_2}(n,E) \ar[u]_-{\iota_{1-}}\,,  
	}
	\\%
	$$
	where the right square diagram is cartesian and tor-independent. Note that 
	$$ \gamma \circ (\tilde{\pi}_{+})* \circ \tilde{\pi}_-^* \circ (\iota_{1-})_* = (\tilde{\alpha}_{1-}^*)_* \circ  \pi_{1+}^* \circ (\pi_{1+})_* \circ \tilde{\alpha}_{1-}^*\,. $$
		By Lemma \ref{lem:pi1inverse}, we see that the composition $\pi_{1+}^* \circ  (\pi_{1+})_*  $ is the identity map. Since $\widetilde{\alpha}_{1-}: Bl_-^{E_1}(n,E) \rightarrow OG^{E_2}_+(n,E)$ is a projective bundle of odd relative dimension,  Theorem \ref{thm:projectivebundleodddimeventwist} implies that $ (\widetilde{\alpha}_{1-})_*  \circ  \widetilde{\alpha}_{1-}^*  $ vanishes.
\end{proof}

\begin{sublemma}\label{sublem:composition-zero-2}
	The composition $ \gamma \circ (\tilde{\pi}_{+})* \circ \tilde{\pi}_-^* \circ \beta$ vanishes.
\end{sublemma}
\begin{proof}[Proof of Sublemma \ref{sublem:composition-zero-2}]
	Form the following diagram:
$$ 
		\xymatrix{
 OG^{E_2}_-(n,E)  \ar@{=}[dd]& \ar[l]_-{\tilde{\alpha}_{1- }} Bl_+^{E_1}(n,E)\ar[r]^-{\pi_{1+}}  \ar@{}[dr] | {\square} &
	 OG^{E_1}_+(n,E)  &  &  \\
			%%%%%%%%%%%%%%%%%%%
	&   H_+(n,E)  \ar@{}[dr] | {\square}   \ar[r]^-{\pi'_{1+}} \ar[u]_-{q_+} &E(n,E)  \ar@{}[dr] | {\square} \ar[r]^-{ \tilde{\pi}_{-}} \ar[u]^-{   \tilde{\pi}_+}   & OG^{E_1}_-(n,E)   \\
			%%%%%%%%%%%%%%%%
 OG^{E_2}_-(n,E)  & \ar[l]_-{\Psi_-} T(n,E)  \ar[d]^-{\Psi_+}  \ar[u]_-{\wp_{+}} \ar[r]^-{\wp_{-}}  &  H_-(n,E)    \ar[u]_-{\pi'_{1-}} \ar[r]^-{q_{-}} & \ar[u]_-{\pi_{1-}} Bl_-^{E_1} (n,E) 
\ar[d]^-{\tilde{\alpha}_{1+ }}  \\
			%%%%%%%%%%%%%%%
 & OG^{E_2}_+(n,E)  & & \ar@{=}[ll] OG^{E_2}_+(n,E) \,, 	} 
	$$  
	where all the square diagrams labelled with $\square$ are defined by fibre products, and we define $\Psi_\mp: = \tilde{\alpha}_{1\mp } \circ q_\pm \circ \wp_\pm$ as depicted in the picture.\ Next, we draw the following picture with diagrams labelled with $\square$ being fibre products \\
$$	\xymatrix{ 	  OG^{E_2}_-(n,E)  \ar@{}[dr] | {\square}   & \ar[l]_-{\tilde{\alpha}_{1- }}  Bl_+^{E_1}(n,E)  \ar@{}[dr] | {\square} &       \ar[l]_-{q_+}    \ar@{}[dr] | {\square}	H_+(n,E)	 & T(n,E)  \ar@/_2.5pc/[lll]^-{\Psi_-} \ar[r]^-{\Psi_+} \ar[l]_-{  \wp_{+}}   &  OG^{E_2}_+(n,E) \ar@{=}[d]  \    \\
	%%%%%%%%%%%%%%%
	E^{E_{1}}(n,E) \ar[u]^-{\tilde{\pi}_{1-}} \ar@/_4pc/[rrrr]_-{\tilde{\pi}_{1+}}   &  \ar[l]_-{\tilde{\alpha}'_{1-}}  E^{E_1}_{Bl_+}(n,E) \ar[u] &	 \ar[u] \ar[l]_-{q'_+}  E^{E_1}_{H_+}(n,E)   & \ar[l]_-{\wp'_+}   \ar@/^2.5pc/[lll]_-{\Psi'_{-}} E^{E_1}_T(n,E) \ar[r]^-{\Phi_+} \ar[u]_-{\tilde{\pi}'_{1-}} & OG^{E_2}_+(n,E)\  ,    }$$
	where we define $\Psi'_-: = \tilde{\alpha}'_{1-} \circ q'_+ \circ \wp'_+$ and $\Phi_+ := \tilde{\pi}_{1+} \circ   \Psi'_- $. Let us  observe the following facts:
 \begin{itemize}
 	\item [(1)] $\pi_{1\pm}$ , $\pi'_{1\pm}$, $\wp_\pm$, $\wp'_+$ are birational and locally complete intersection.  
 	\item [(2)] $\tilde{\pi}_\pm$, $q_\pm $, $q'_+$, $\tilde{\alpha}_{1\pm}$, $\tilde{\alpha}'_{1-}$ are projective bundles of relative dimension $n-2$ (an odd number).
 	 \item [(3)] $\tilde{\pi}_{1\pm}$ are projective bundles of relative dimension $n-3$ (an even number).
 \end{itemize}
All these claims follow from Theorem \ref{theo:blowup} and the properties of fibre products.\ 

Notice that all square diagrams labelled with $\square$ are tor-independent. Only the square labelled with $\square$ in the left lower corner of the first displayed diagram requires an explanation. Since $\pi'_{1\pm}$, $\wp_\pm$ are birational and locally complete intersection, we are reduced to show that a fibre square of four regular immersions of the same codimension is tor-independent.\ The proof for this fact follows from the argument of \cite[Lemma 22]{Fasel09}. By repeatedly applying the base change formula, we see that $$\gamma \circ (\tilde{\pi}_{+})* \circ \tilde{\pi}_-^* \circ \beta = (\Psi_-)_* \circ \Psi_+^*.$$
Since $\tilde{\pi}_{1-}^*$ is an isomorphism (cf. Theorem \ref{thm:projectivebundleevendimoddtwist})), it is enough to show that  $ \tilde{\pi}_{1-}^* \circ (\Psi_-)_* \circ \Psi_+^* = 0$. Note that 
$$\tilde{\pi}_{1-}^* \circ (\Psi_-)_* \circ \Psi_+^* = (\Psi'_-)_* \circ \Phi_+^* =  (\Psi'_-)_* \circ  \Psi_-'^* \circ  \tilde{\pi}_{1+}^*, $$ and
$$ (\Psi'_-)_* \circ  \Psi_-'^* = (\tilde{\alpha}'_{1-})_* \circ (q'_+)_* \circ (\wp'_+)_*  \circ \wp_+'^* \circ q_+'^* \circ  \tilde{\alpha}_{1-}'^*. $$
By Lemma \ref{lem:projblow}, we see $(\wp'_+)_*  \circ \wp_+'^*  = \textnormal{id}$, and the result follows from $(q'_+)_* \circ q_+'^* = 0 $ (cf.\ Theorem \ref{thm:projectivebundleodddimeventwist}). 
\end{proof}

\noindent \textbf{Proof of Theorem \ref{thm:odd dimension case} (continue).}	Let $\theta_+:OG^{E_2}_+(n,E)  \to OG_+(n,E)  $
	denote the composition $\iota_+ \circ \iota_{1+} $. By the above sublemmas, we have constructed the following short exact sequence
	\[ \xymatrix{   0  \ar[r] &  W^{i-(2n-3)}(OG^{E_2}_+(n,E), L \otimes \omega_{\theta_+})  \ar[r]^-{(\theta_+)_*} &     W^i(OG_+(n,E),L)  \ar[r]^-{\Omega_+} &    W^i(OG_+^{E_2}(n,E),L) \ar[r] & 0\,, }       \]
	where $L$ is a line bundle over the base $S$, and $\Omega_+ : =  (\alpha_{1+}^*)^{-1} \circ v_{1-}^* \circ (\alpha_{-}^*)^{-1} \circ v_+^*$. Note here that $[\omega_{\theta_+}] = 1$ in $\Pic(OG^{E_2}_+(n,E))/2$.  This formula follows from Proposition \ref{prop:canonicalbundle} and Remark \ref{rmk:canonicalbundleinn-1}. 
	
	It is clear that $\Omega_+$ is a morphism of graded $W^\tot(S)$-modules, because it is the composition of morphisms of graded $W^\tot(S)$-modules. Using the projection formula, we obtain the following short exact sequence of graded $W^\tot(S)$-modules
	\[ \xymatrix{   0  \ar[r] &  W^{\tot}(OG^{E_2}_+(n,E))\big[-(2n-3)\big]  \ar[r]^-{\theta_*} &     W^\tot(OG_+(n,E))  \ar[r]^-{\Omega_+} &    W^\tot(OG_+^{E_2}(n,E)) \ar[r] & 0 \,. }       \]
	Since $OG_+^{E_2}(n,E) \cong OG_+(n-2, E^2)$, by induction on $n$ we can assume that $W^\tot(OG_+^{E_2}(n,E)) $ is a graded projective $W^\tot(S)$-module. In fact, we can reduce to the $n=1$ case and $OG_+(1,E) \cong S$, therefore this exact sequence splits.
\end{proof}

\subsection{Proof of Theorem \ref{thm:main-theorem}}\label{sec:maintheorem} Note that trivial hyperbolic bundles always admit canonical complete flags. The statement of Theorem \ref{thm:main-theorem} for $n$ odd follows inductively from Theorem \ref{thm:odd dimension case} and results in Section \ref{sect:closed embeds}. For $n$ even, in view of Proposition \ref{thm:evendim-eventwist} and Lemma \ref{lem:OGE2splits}, one has the following isomorphism of graded $W^\tot(S)$-modules:
	\begin{equation}\label{eq:neven}
	W^\tot(OG_+(n,E)) \cong W^{\tot}(OG^{E_1}_+(n,E))\big[-(n-1),\ \det(\tilde{E}_n)\big] \oplus W^\tot(OG^{E_1}_+(n,E)) \,,
	\end{equation}
 which allows us to reduce to the odd case. By the assumption, $\det(\tilde{E}_n)$ is trivial, and the result follows. 
\begin{remark}
	For odd $n$, one can obtain a visual interpretation of Theorem \ref{thm:odd dimension case}. In fact, in this case  one has that even shifted Young diagrams must either have the first two rows of maximal length (as in those appearing in the first line of  Example \ref{exam:og(7)}) or two empty columns on the right (in the second line). The number of boxes in the first two rows is precisely $2n-3$, which, not by accident, also appears as the shift in the statement of Theorem \ref{thm:odd dimension case}.
	
	For $n$ even there are two classes of even shifted Young diagrams: those whose first row is full  and those whose last column is empty. Note that this fact can be viewed as a diagrammatic counterpart to the recursive description given in Formula (\ref{eq:neven}). 
\end{remark}

\begin{appendix}
	\section{Euler Classes and Projective bundles}
	\begin{center}
		Heng Xie
	\end{center}
	
	The projective bundle theorem for Witt groups was independently proved by Walter \cite{Walter} and Nenashev \cite{Ne}. In the main body of the paper we need to apply two theorems (Theorem \ref{thm:projectivebundleevendimoddtwist} and \ref{thm:projectivebundleodddimeventwist} in this appendix), which were proved as Theorem 1.2 and Theorem 1.4 in the unpublished \cite{Walter}. Although the approach by Nenashev is closer to our needs, his description of the underlying maps appearing in the projective bundle theorem is not very explicit. The aim of this appendix is to highlight some maps that in \cite{Ne} are sort of hidden and simultaneously solve the lack of a published reference for the theorems that we need.\  
	It is worth mentioning that Fasel \cite{Fasel13} has written down explicitly the underlying maps for the projective bundle theorem of $\I^j$-cohomology, which can be compared with our Theorem \ref{thm:projectivebundleevendimoddtwist} and \ref{thm:projectivebundleodddimeventwist}.

	\subsection{Projective bundles}\label{sec:projective bundles} Let $\rho \colon \mathcal{E} \rightarrow X$ be a vector bundle of rank $r+1$ over a regular scheme $X$ with $\frac{1}{2} \in \SO_X$. Let $s \colon X \rightarrow \mathcal{E}$ be the zero section of the bundle $\mathcal{E}$. Let $q\colon  \mathbb{P}(\mathcal{E}) \rightarrow X$ be the projective bundle associated to the vector bundle $\mathcal{E}$. Then, there is an exact sequence of vector bundles on $\mathbb{P}(\mathcal{E})$
	\[\xymatrix{ 0\ar[r] &\SO_{\mathbb{P}(\mathcal{E})}(-1) \ar[r] & q^* \mathcal{E} \ar[r] & \mathcal{Q} \ar[r] & 0 \,,}\]
	where $\mathcal{Q}$ is the universal quotient bundle.\ Consider the projective bundle $q': \mathbb{P}(\mathcal{E} \oplus \SO) \to X $.  The canonical split exact sequence $0 \rightarrow \SO \rightarrow \mathcal{E} \oplus \SO \rightarrow \mathcal{E} \rightarrow 0$ induces two closed embeddings
	$\iota \colon \mathbb{P}(\SO) = X \rightarrow \mathbb{P}(\mathcal{E} \oplus \SO) \textnormal{ and } \iota':\mathbb{P}(\mathcal{E}) \rightarrow \mathbb{P}(\mathcal{E} \oplus \SO) .$
	Let 
	$ v: U    \hookrightarrow \mathbb{P}(\mathcal{E} \oplus \SO)$ be the associated open complement of $\iota$. There is a canonical map $\alpha \colon U \rightarrow \mathbb{P}(\mathcal{E})$, which is a vector bundle. 
	
	Let $\mathcal{W}_{\mathbb{P}(\mathcal{E})}$ denote the vector bundle $\SO_{\mathbb{P}(\mathcal{E})}(-1) \oplus \SO_{\mathbb{P}(\mathcal{E})}  $ over $\mathbb{P}(\mathcal{E})$. Consider the projective bundle $\mathbb{P}(\mathcal{W}_{\mathbb{P}(\mathcal{E})})$ over $\mathbb{P}(\mathcal{E})$ equipped with the projection $\tilde{\alpha}: \mathbb{P}(\mathcal{W}_{\mathbb{P}(\mathcal{E})} ) \to \mathbb{P}(\mathcal{E}) $. 	By the universal property of projective bundles, there is a morphism $\pi: \mathbb{P}(\mathcal{W}_{\mathbb{P}(\mathcal{E})}) \to \mathbb{P}(\mathcal{E} \oplus \SO)$ defined by the rank one subbundle $\SO_{\mathbb{P}(\mathcal{W}_{\mathbb{P}(\mathcal{E})})}(-1) \subset \tilde{\alpha}^* \mathcal{W}_{\mathbb{P}(\mathcal{E})} \subset \mathcal{E}_{\mathbb{P}(\mathcal{W}_{\mathbb{P}(\mathcal{E})})} \oplus \SO_{\mathbb{P}(\mathcal{W}_{\mathbb{P}(\mathcal{E})})}$. One can form the following cartesian diagram
	
	\begin{equation}\label{eq:blowupPOPEO}\xymatrix{ \mathbb{P}(\SO) \ar@{^{(}->}[r]^-{\iota} & \mathbb{P}(\mathcal{E} \oplus \SO) \\
			\mathbb{P}(\SO_{\mathbb{P}(\mathcal{E})}) \ar[u]^-{\tilde{\pi}} \ar@{^{(}->}[r]^-{\tilde{\iota}} & \mathbb{P}(\mathcal{W}_{\mathbb{P}(\mathcal{E})})  \ar[u]^-{\pi} .
	}\end{equation}
		\begin{Prop}\label{prop:blowupproj}
		The projective bundle  $\mathbb{P}(\mathcal{W}_{\mathbb{P}(\mathcal{E})})$ is the blow-up of $ \mathbb{P}(\mathcal{E} \oplus \SO)$ along $\mathbb{P}(\SO)$ with exceptional fiber $	\mathbb{P}(\SO_{\mathbb{P}(\mathcal{E})}) $. Moreover,  $U \cong  \mathbb{P}(\mathcal{W}_{\mathbb{P}(\mathcal{E})})  -  	\mathbb{P}(\SO_{\mathbb{P}(\mathcal{E})}) $ is the total space of $\SO_{\mathbb{P}(\mathcal{E})}(1)$ over $\mathbb{P}(\mathcal{E})$.  
	\end{Prop}
The whole setting is depicted in the following picture
	\begin{equation}\label{eq:blowupPOPEO}\xymatrix{ \mathbb{P}(\SO) \ar@{^{(}->}[r]^-{\iota} & \mathbb{P}(\mathcal{E} \oplus \SO) & \ar[l]_-{v} \ar[dl]_-{\tilde{v}} U \ar[d]_-{\alpha} \\
		\mathbb{P}(\SO_{\mathbb{P}(\mathcal{E})}) \ar[u]^-{\tilde{\pi}} \ar@{^{(}->}[r]^-{\tilde{\iota}} & \mathbb{P}(\mathcal{W}_{\mathbb{P}(\mathcal{E})})  \ar[u]^-{\pi} \ar[r]^-{\tilde{\alpha}} & \mathbb{P}(\mathcal{E}) \,,
}\end{equation}
where $\alpha: = \tilde{\alpha} \circ \tilde{v}$. Set $B: = \mathbb{P}(\mathcal{W}_{\mathbb{P}(\mathcal{E})}) $ and $ E: =	\mathbb{P}(\SO_{\mathbb{P}(\mathcal{E})}) $. 
 \begin{proof}
By the compatibility of blow-ups and pullbacks, one can reduce to the case when $X$ is affine, and suppose that $\mathcal{E}$ is free. This case is done in \cite[Chapter V, Example 2.11.4]{Hartshorne}
 \end{proof}

	\subsection{Pushforward and pullback}
	Let $f: X \to Y$ be a proper morphism, and let $\LC$ be a line bundle on $Y$. Recall two functorial maps from \cite{CH11}: the pullback
	$$  f^* \colon W^i(Y, \mathcal{L}) \rightarrow W^i(X,  f^*\mathcal{L}) $$ and the pushforward $$f_* \colon  W^{i+\dim X}(X, f^* \mathcal{L} \otimes \omega_{X/Y}) \rightarrow  W^{i+ \dim Y}(Y, \mathcal{L}). $$ 

For properties of pushforward and pullback, we refer to \cite{CH11}.\ The following lemma is useful throughout several arguments in our paper: it is an analog of \cite[Proposition 6.7 (b)]{Fu}. 
\begin{lem}\label{lem:projblow}
Let $\pi : B \to Y$ be a proper birational morphism. Suppose that $ [\omega_\pi] = 1$ in $\Pic(B)/2$, and that $\pi_*(1_B) =1_Y$ in Witt groups. Then, $ \pi_* \circ  \mathrm{per} \circ  \pi^* (y) =  y $ for any $y \in W^i(Y,\LC)$. 
	\end{lem}
	\begin{proof}
	By the projection formula \cite[Theorem 6.5]{CH11}, we see that $\pi_* \circ \mathrm{per} \circ \pi^*(y) = \pi_* \circ \mathrm{per} \circ (1_B \cdot \pi^*(y)) = \pi_*(1_B) \cdot y = y$. 
	\end{proof}
\begin{remark} If $B$ is the blow up of $Y$ along a regular closed subscheme $Z$, then the condition $\pi_*(1_B) = 1_Y$ is satisfied, cf.\ \cite[Proposition 3.15]{BC12}. 
\end{remark}
\subsection{Pushforward and projective bundles} For the case of the projective bundle $q: \mathbb{P}(\mathcal{E})\to X$, the exact sequence 
	%Hartshorne Theorem 8.13
	\[\xymatrix{ 0\ar[r] &\Omega_{\mathbb{P}(\mathcal{E})/X} \ar[r] & q^* \mathcal{E}^\vee \otimes \SO(-1) \ar[r] & \SO_{\mathbb{P}(\mathcal{E})} \ar[r] & 0}\]
	yields a canonical isomorphism $\omega_{\mathbb{P}(\mathcal{E})/X} : = \det(\Omega_{\mathbb{P}(\mathcal{E})/X}) \cong q^* \det \mathcal{E}^\vee \otimes \SO(-r-1).$
	Therefore, we are able to rewrite the pushforward as $$q_* \colon  W^i(\mathbb{P}(\mathcal{E}), q^* \mathcal{L} \otimes \det\mathcal{E}^\vee \otimes \SO(-r-1)) \rightarrow  W^{i-r}(X, \mathcal{L}).$$

	\subsection{Euler class} 
	Let $\mathcal{L}$ be a line bundle on $X$. Let $\rho:\mathcal{V} \to X$ be a vector bundle of rank $d$ and let $\nu: X \to \mathcal{V}$ be its zero section. The \textit{Euler class} of $\mathcal{V}$ is the following composition of maps 
	\[\xymatrix{ W^i(X, \mathcal{L}) \ar[r]^-{\nu_*} & W^{i+d} (\mathcal{V}, \rho^*(\det \mathcal{V}^\vee \otimes \mathcal{L})) \ar[r]^-{(\rho^*)^{-1}} & W^{i+d} (X, \det\mathcal{V}^\vee \otimes \mathcal{L}) \,,} \]
	which shall be denoted by $e(\mathcal{V})$ (cf. \cite{Fasel09} or \cite[Section 3]{Fasel13} for details).\ For the relations between pushforwards, pullbacks and Euler classes, we refer to \cite{Fasel09} (see also \cite{Fasel13}), and the excess intersection formula plays an important role in this paper.

	\subsection{Witt groups of projective bundles}
	We use notations as in Section \ref{sec:projective bundles}. 
	\begin{theo}[Theorem 1.2 \cite{Walter}]\label{thm:projectivebundleevendimoddtwist}
		Assume that $r$ is an even number. The maps of groups
		\[ \xymatrix{  q_*\circ \mathrm{per}  \colon  W^i(\mathbb{P}(\mathcal{E}), \SO(-1)) \ar@<0.5ex>[r]  & \ar@<0.5ex>[l] W^{i-r}(X, \det \mathcal{E}) \colon    e(\mathcal{Q})\circ q^*  }\]
		are inverse isomorphisms.
	\end{theo}
	\begin{proof}
	By Proposition \ref{prop:blowupproj}, \cite[Setup 1.1 and Hypothesis 1.2]{BC} is satisfied.\ By \cite{CH11}, we have a pushforward map $ \pi_*: W^i(B, \omega_\pi) \rightarrow W^i(\mathbb{P}(\mathcal{E} \oplus \SO) ) $. Note that by \cite[Proposition 2.1]{BC} one has $\omega_\pi = (0, r) \in  \Pic(\mathbb{P}(\mathcal{E} \oplus \SO ))\oplus \mathbb{Z} \cong \Pic(B) $. Since $r$ is even, we have a periodicity isomorphism $\mathrm{per}: W^i(B) \rightarrow W^i(B, \omega_\pi)$. Form the following morphisms of short exact sequences coming from localization sequences
		\[ \xymatrix{
			0 \ar[r] & W^{i-1}(E, \omega_{\tilde{\iota}} ) \ar[d]^-{\mathrm{per}} \ar[r]^-{(\tilde{\iota})_*}  & W^i(B)      \ar[r]^-{\tilde{v}^*} \ar[d]^-{ \mathrm{per}} & W^i(U) \ar[d]^-{=} \ar[r]& 0
			\\ %%%%%%%%%%%%
			0 \ar[r] & W^{i-1}(E, \omega_{\tilde{\iota}} \otimes \tilde{\iota}^* \omega_\pi ) \ar[d]^-{\tilde{\pi}_* } \ar[r]^-{(\tilde{\iota})_*}  & W^i(B, \omega_\pi)      \ar[r]^-{\tilde{v}^*} \ar[d]^-{\pi_*} & W^i(U) \ar[d]^-{=} \ar[r]& 0
			\\ %%%%%%%%%%%%
			0 \ar[r] & W^{i-1-r}(X, \omega_{\iota}  ) \ar[d]^-{ e(\mathcal{Q}) \circ \tilde{\pi}^*} \ar[r]^-{(\iota)_*}  & W^i(\mathbb{P}(\mathcal{E} \oplus\SO))      \ar[r]^-{v^*} \ar[d]^-{\pi^*} & W^i(U)  \ar[d]^-{=} \ar[r]& 0
			\\%%%%%%%%
			0 \ar[r] & W^{i-1}(E, \omega_{\tilde{\iota}})  \ar[r]^-{(\tilde{\iota})_*}  & W^i(B)      \ar[r]^-{\tilde{v}^*}  & W^i(U) \ar[r]& 0  .}\]
		Here all the squares are commutative and in particular the commutativity of the lower left diagram follows by the excess intersection formula (cf. \cite{Fasel09}). (Note here that $\pi$ is of finite Tor-dimension, since $\pi$ is the composition $B \hookrightarrow \mathbb{P}(\mathcal{E} \oplus \SO) \times \mathbb{P}(\mathcal{E}) \to  \mathbb{P}(\mathcal{E} \oplus \SO)$, where the first morphism is a regular immersion and the second morphism is the projection.)\ All the short exact sequences are split, by the computation of Witt groups of projective bundles (cf. \cite{Ne}). Recall the isomorphism $\omega_{\tilde{\iota}} \cong \SO_{E}(-1)$ (cf. \cite[Appendix A]{BC}). The excess normal bundle of the left cartesian square of Diagram \ref{eq:blowupPOPEO} is the universal quotient bundle $\mathcal{Q}$ on $ \mathbb{P}(\mathcal{E}) \cong E$. Theorem \ref{thm:projectivebundleevendimoddtwist} will follow if we can prove that 
		$  \pi^* \circ \pi_* \circ \mathrm{per} = \mathrm{id} \textnormal{ and }  \pi_* \circ \mathrm{per} \circ \pi^* = \mathrm{id}. $

Note that $ \pi_* \circ \mathrm{per} \circ \pi^* = \mathrm{id}$ by Lemma \ref{lem:projblow}. It is therefore enough to prove that $\pi^*$ is an isomorphism. Form the cartesian diagram
		\begin{equation}\label{eq:PEPEO}
			\xymatrix{ \mathbb{P}(\mathcal{E})  \ar[r]^-{\iota'} & \mathbb{P}(\mathcal{E} \oplus \SO) 
				\\%%%%%%%%%
				\mathbb{P}(\mathcal{\SO}_{\mathbb{P}(\mathcal{E})}(-1))  \ar[u]^-{g}_-{\cong} \ar[r]^-{\tilde{\iota}'} &  B \ar[u]^{\pi},
				\\ %%%%%%%%
		}\end{equation}
 where $g$ is an isomorphism since $\iota'$ factors through $U$. By applying the localization theorem to $\iota'$ and $\tilde{\iota}'$, we obtain the following morphism of split exact sequences of Witt groups (cf. Lemma \ref{lem:chowmoving})
		\begin{equation}\label{eq:projectivebundle-even-bl} \xymatrix{ 
				0 \ar[r] & W^{i-1}(\mathbb{P}(\mathcal{E}), \omega_{\iota'}  ) \ar[d]^-{g^*}_-{\cong} \ar[r]^-{(\iota')_*}  & W^i(\mathbb{P}(\mathcal{E} \oplus\SO))      \ar[r]^-{\iota'^*} \ar[d]^-{\pi^*} & W^i(\mathbb{P}(\mathcal{E}))  \ar[d]^-{g^*}_-{\cong} \ar[r]& 0
				\\%%%%%%%%
				0 \ar[r] & W^{i-1}(\mathbb{P}(\mathcal{\SO}_{\mathbb{P}(\mathcal{E})}(-1)) , \omega_{\tilde{\iota}'} ) \ar[r]^-{(\tilde{\iota}')_*}  & W^i(B)      \ar[r]^-{\tilde{\iota}'^*} & W^i(\mathbb{P}(\mathcal{\SO}_{\mathbb{P}(\mathcal{E})}(-1)) )  \ar[r]& 0. }\end{equation}
The left square is commutative by \cite{Fasel09} (recall that $\pi$ is of finite Tor-dimension and note that the excess normal bundle is trivial).\ This allows us to conclude that the map $\pi^*$ is an isomorphism.
	\end{proof}
	
	\begin{lem}\label{lem:chowmoving}
Suppose that $0 \to \mathcal{E} \to \mathcal{E}' \to \mathcal{L} \to 0$ is an exact sequence with $\mathcal{L}$ a line bundle on $X$. Let $\iota: \mathbb{P}(\mathcal{E}) \to \mathbb{P}(\mathcal{E}')$ be the induced closed embedding. Then, the sequence
		\[ \xymatrix{0 \ar[r] & W^{i-1}(\mathbb{P}(\mathcal{E}), \omega_{\iota}) \ar[r]^-{\iota_*} &W^{i}(\mathbb{P}(\mathcal{E}')) \ar[r]^-{\iota^*} & W^i(\mathbb{P}(\mathcal{E})) \ar[r] & 0   } \]
		is split exact. 
	\end{lem}
	\begin{proof}
		Let $q: \mathbb{P}(\mathcal{E}) \to X$ and $q':\mathbb{P}(\mathcal{E}') \to X$ be the projections. Let $U : = \mathbb{P}(\mathcal{E}') - \mathbb{P}(\mathcal{E}) $.\ Note that $U$ is an affine bundle over $X$, and let $p: U \to X$ be the projection. By \cite[Section 5]{Ne}, the exact sequence
		\[ \xymatrix{0 \ar[r] & W^{i-1}(\mathbb{P}(\mathcal{E}), \omega_{\iota}) \ar[r]^-{\iota_*} &W^{i}(\mathbb{P}(\mathcal{E}')) \ar[r]^-{v^*} & W^{i}(U) \ar[r] & 0   } \]
		splits on the right by the map $ q'^* \circ (p^*)^{-1}$. By \cite[Theorem 3.6]{Ne}, we also know that $q^*: W^i(X) \to W^i(\mathbb{P}(\mathcal{E}))$ is an isomorphism. Therefore one has  $   \iota^* \circ   q'^* \circ (q^*)^{-1} =  \mathrm{id}$ and $\iota^*$ is split surjective. 
		
	It is enough to show that $\iota^*\circ \iota_* =0 $.  By the self-intersection formula \cite[Theorem 33]{Fasel09}, we see that $\iota^* \circ \iota_* = e(\omega_{\iota}) $. Note that the Euler class $e(\omega_{\iota})$ is hyperbolic \cite[Proposition 14]{FS}, since $\omega_{\iota}$ is a line bundle. Therefore,  $\iota^* \circ \iota_* =0 $. The result follows. 
	\end{proof}
	\begin{remark}	
		%	Nenashev \cite[Theorem 5.1]{Ne} 
		%	considers two long exact sequences
		%
		%\[     \xymatrix{ \cdots \ar[r] &  W^i_{\mathbb{P}(\SO)}(\mathbb{P}(\mathcal{E}\oplus \SO)) \ar[r] & W^i(\mathbb{P}(\mathcal{E} \oplus \SO) )  \ar[r]  &  W^i(U)  \ar[r] &   W^{i+1}_{\mathbb{P}(\SO)}(\mathbb{P}(\mathcal{E}\oplus \SO))   \ar[r]  & \cdots }\]
		%and
		%\[     \xymatrix{ \cdots \ar[r] &  W^i_{\mathbb{P}(\mathcal{E})}(\mathbb{P}(\mathcal{E} \oplus \SO) ) \ar[r] & W^i(\mathbb{P}(\mathcal{E} \oplus \SO) )  \ar[r]  &  W^i(U' )  \ar[r] &   W^{i+1}_{\mathbb{P}(\mathcal{E} )}(\mathbb{P}(\mathcal{E} \oplus \SO))   \ar[r]  & \cdots }. \]
		
		%Since the canonical maps $\alpha: U \rightarrow\mathbb{P}(\mathcal{E})  $
		%and
		%$\alpha':U'  \rightarrow \mathbb{P}(\SO) $ are both affine bundles, the above two long exact sequences split, by using that $q^*: W^i(X) \rightarrow W^i(\mathbb{P}(\mathcal{E})) $ is an isomorphism. Since $W^i_{\mathbb{P}(\SO)}(\mathbb{P}(\mathcal{E} \oplus \SO)) \cong W^{i-(r+1)} (X, \det \mathcal{E}) $ and $W^{i-1}(\mathbb{P}(\mathcal{E}),\SO(1)) \cong W^i_{\mathbb{P}(\mathcal{E})}(\mathbb{P}(\mathcal{E} \oplus \SO))$, 
		%	by the d\'{e}vissage theorem and the  two periodicity on the twists, 
		In \cite[Theorem 5.1]{Ne} Nenashev is able to show that $W^{i-r} (X, \det \mathcal{E})$ and $W^i(\mathbb{P}(\mathcal{E}), \SO(1))$ are isomorphic, however it is more subtle to gain an understanding of the intermediate maps whose composition gives rise to this isomorphism, and this is the precisely the purpose of Theorem \ref{thm:projectivebundleevendimoddtwist}. 
	\end{remark}
	
	\begin{theo}[Theorem 1.4 \cite{Walter}]\label{thm:projectivebundleodddimeventwist}
		Assume that $r$ is an odd number. The sequence of maps of groups
		\[ \xymatrix{    \cdots  \ar[r] & W^i(X) \ar[r]^-{q^*} & W^i(\mathbb{P}(\mathcal{E})) \ar[r]^-{q_* \circ \mathrm{per}} & W^{i-r}(X, \det \mathcal{E}) \ar[r]^-{e(\mathcal{E})}  & W^{i+1}(X) \ar[r] & \cdots }\]
		is exact.
	\end{theo}
	\begin{proof} Let $p: \mathbb{P}(\mathcal{E} \oplus \SO) \to X $ be the canonical projection. Consider the exact sequence
		\begin{equation}\label{eq:exactseqonPOE}\xymatrix{ 0\ar[r] &\SO_{\mathbb{P}(E\oplus\SO)}(-1) \ar[r] & p^* (\mathcal{E}\oplus \SO) \ar[r] & \mathcal{G} \ar[r] & 0 }\end{equation}
		with $\mathcal{G}$ the universal quotient bundle and form the following ladder diagram.
		\begin{equation}\label{eq:Podd} \xymatrix{   \cdots  \ar[r] & W^i(U') \ar[r]^-{\partial} & W^{i+1}_{\mathbb{P}(\mathcal{E})}(\mathbb{P}(\mathcal{E}\oplus \SO), \SO(-1)) \ar[r] & W^{i+1}(\mathbb{P}(\mathcal{E}\oplus \SO) , \SO(-1) )\ar[r]^-{v'^*}  & W^{i+1}(U') \ar[r] & \cdots &&
				\\  %%%%%%%%%%%%%%%%
				\cdots  \ar[r] & W^i(X) \ar[u]^-{\alpha^*} \ar[r]^-{q^*} & W^i(\mathbb{P}(\mathcal{E})) \ar[u]^-{\iota'_*} \ar[r]^-{q_* \circ \mathrm{per}} & W^{i-r}(X, \det \mathcal{E}) \ar[u]^-{e(\mathcal{G}) \circ p^*} \ar[r]^-{e(\mathcal{E})}  & W^{i+1}(X) \ar[u]^-{\alpha^*} \ar[r] & \cdots }\end{equation}
		Note that all the vertical maps in (\ref{eq:Podd}) are isomorphisms and that the upper sequence is exact. It is now enough to prove that all squares  in Diagram \ref{eq:Podd} are commutative.  
		
		By \cite[Lemma 4.2 (B)]{BC}, we have the following commutative diagram
		\[\xymatrix{W^i(U') \ar[r]^-{\partial} & W^{i+1}_{\mathbb{P}(\mathcal{E})}(\mathbb{P}(\mathcal{E}\oplus \SO), \SO(-1))  &
			\\%%%%%%%%%%%%
			W^i(\mathbb{P}(\mathcal{E} \oplus \SO)) \ar[u]^-{v'^*} \ar[r]^-{\iota'^*} & W^i(\mathbb{P}(\mathcal{E}))\ar[u] &\phantom{a}\hspace{- 1.5 cm} ,}\]
		which explains the commutativity of the left square in (\ref{eq:Podd}). In order to conclude that the middle square is commutative, we use  that, by Theorem \ref{thm:projectivebundleevendimoddtwist}, the map $p_* \circ \mathrm{per}$ is the inverse of $ e(\mathcal{G})\circ p^*$, so that the commutativity follows from $  p_* \circ \iota'_* = q_* $. Finally, in order to see that the right square is commutative, we observe that $\alpha^* \iota^* = v'^*$ on the level of Picard groups. By applying the pullback $\iota^*$ to (\ref{eq:exactseqonPOE}), we see that $\iota^* \mathcal{G} \cong \mathcal{E}$ and therefore $v'^* \mathcal{G} = \alpha^* \mathcal{E}$. Now one has $$v'^* e(\mathcal{G}) p^* = e(v'^*\mathcal{G})v'^* p^* = e(v'^*\mathcal{G})\alpha^* =   e(\alpha^*\mathcal{E})\alpha^* = \alpha^* e(\mathcal{E}).\qedhere $$
	\end{proof}

\noindent \textbf{Acknowledgement}. We would like to thank Nicolas Perrin and Marcus Zibrowius for comments and helpful discussions.\ We are grateful to the anonymous referee whose careful reading and very helpful suggestions greatly improved the readability of this work.\ HX would like to acknowledge support from the EPSRC Grant EP/M001113/1, the DFG priority programme 1786 and the Fundamental Research Funds from the Central Universities, Sun Yat-sen University 34000-31610294.\ HX would also like to thank hospitality of Max-Planck-Institute in Bonn. Research for this publication was conducted in the framework of the DFG Research Training Group 2240: Algebro-Geometric Methods in Algebra, Arithmetic and Topology, while TH and HX were affiliated to the Bergische Universit\"{a}t Wuppertal.

\end{appendix}


\begin{thebibliography}{0}
	
	\bibitem{Balmer}
	P.~ Balmer, Witt groups. Handbook of K-theory. Vol. 1, 2, 539–576, Springer, Berlin, 2005.
	
	\bibitem{Balmer 01}
	P. Balmer, Witt cohomology, Mayer-Vietoris, homotopy invariance and the Gersten conjecture, K-theory 23 (2001) 15-30.
	
	\bibitem{BC}
	P.~Balmer
	and
	B.~Calm\`{e}s,
	{ Geometric description of the connecting homomorphism for Witt groups,\/}
	Documenta Mathematica, 14 (2009), 525-550.
	
	\bibitem{BC12}
	P.~Balmer
	and
	B.~Calm\`{e}s,
	Witt groups of Grassmann varieties,
	Journal of Algebraic Geometry (4) 21 (2012), pp. 601-642.
	
	\bibitem{BG}
	P. Balmer
	and
	S. Gille,
	Koszul complexes and symmetric forms over the punctured affine space, Proceedings of the London Mathematical Society (2) 91 (2005), pp. 273-299.
	
	\bibitem{Bourbaki}
	N. Bourbaki, Alg\`{e}bre Chaitre 9: Fromes sesquilin\'{e}aries, \'{E}l\'{e}ment de Math\'{e}matiques, Springer. 
	
	\bibitem{Calmeselt}
	B. Calm\`{e}s, E. Dotto, J. Harpaz, F. Hebestreit, M. Land, K. Moi, D. Nardin, T. Nikolaus, and W. Steimle, Hermitian K-theory for stable $\infty$-categories, https://arxiv.org/abs/2009.07223, 2021.
	
	\bibitem{Calmes-Fasel}
	B.~Calm\`{e}s and J. Fasel, Trivial Witt groups of flag varieties. 
Journal of Pure and Applied Algebra (2) 216 (2012), 404-406.
	
	\bibitem{CH11}
	B.~Calm\`{e}s
	and
	J.~Hornbostel,
	{  Push-forwards for Witt groups of schemes,\/}
	Comment. Math. Helv. 86 (2011), 437–468.

	\bibitem{Fasel09}
	J. Fasel,
	{   The excess intersection formula for Grothendieck–Witt groups,\/}
	Manuscripta Mathematica 130 (2009).
	
	\bibitem{Fasel13}
	J. Fasel,
	{   The projective bundle theorem for $I^j$-cohomology,\/}
	Journal of K-theory (2) 11 (2013), 413-464.
	
	\bibitem{FS}
	J. Fasel and V. Srinivas, 
	\emph{Chow-Witt groups and Grothendieck-Witt groups of regular schemes},
	Advances in Mathematics
	221 (2009), 302-329.

	
	\bibitem{Fu}
	W. Fulton,
	{   Intersection Theory,\/}
	Springer (1978).
	
	\bibitem{FP}
W. Fulton
and
P. Pragacz,
{   Schubert Varieties and Degeneracy Loci,\/}
Lecture Notes in Mathematics 1689, Springer (1998).


	
	\bibitem{Gille03}
	S. Gille, Homotopy invariance of coherent Witt groups, Math. Z. 244 (2003), 211–233.
	
	
	\bibitem{Gille}
	S. Gille, The general d\'evissage theorem for Witt groups, Arch. Math. 88 (2007), 333-343.
	
	\bibitem{GW}
	G. Ulrich
	and
	W. Torsten,
	Algebraic Geometry, Part I: Schemes. With Examples and Exercises,
	Advanced Lectures in Mathematics, Springer 2010.
	
	
	\bibitem{HXZ20}
	J. Hornbostel, H. Xie, and M. Zibrowius, Chow-Witt rings of split quadrics,
	In Motivic homotopy theory and refined enumerative geometry, Contemporary Mathematics, AMS series, 745 (2020), 123-162.
	
	
	\bibitem{Hartshorne}
	R. Hartshorne,
	Algebraic geometry .
	Graduate Texts in Mathematics 52. Springer-Verlag, New York, 1977.
	
	\bibitem{KW20}
	M. Karoubi
	and
	C. Weibel,
	The Witt group of real surfaces. K-theory in algebra, analysis and topology, 157–193, Contemp. Math., 749, Amer. Math. Soc., Providence, RI, 2020.
	
	\bibitem{KW16}
	M. Karoubi,	M. Schlichting,	and C. Weibel,
	The Witt group of real algebraic varieties.
	J. Topol. (4) 9 (2016), 1257–1302.
	
	\bibitem{KSW20}
	M. Karoubi,
	M. Schlichting,
	and
	C. Weibel,
	Grothendieck-Witt groups of some singular schemes,
	Proc. Lond. Math. Soc. (4) 122 (2021), 521-536.
	
	\bibitem{Kne77}
	M. Knebusch, Symmetric bilinear forms over algebraic varieties. Conference on Quadratic Forms-1976 (Proc. Conf., Queen's Univ., Kingston, Ont., 1976), pp. 103–283. Queen's Papers in Pure and Appl. Math., No. 46, Queen's Univ., Kingston, Ont., 1977. 
	
	\bibitem{Knus-Ojanguren}
	M. A. Knus and M. Ojanguren, The Clifford algebra of a metabolic space. Arch. Math. (Basel) (5) 56 (1991), 440-445.
	
	
	\bibitem{LM07}
	M. Levine and F. Morel, Algebraic cobordism, Springer Monographs in Mathematics, Springer,
	Berlin, 2007.
	
	\bibitem{Mumford}
	D. Mumford, Theta characteristics of an algebraic curve, Annales scientifiques de l’É.N.S., (2) 4 (1971), p. 181-192.
	
	\bibitem{Ne}
	A. Nenashev,
	{  On the Witt groups of projective bundles and split quadrics: Geometric reasoning \/},
	J. $K$-theory 3 (2009), 533-546.
	
	\bibitem{Walter}
	C. Walter,
	{  Grothendieck-Witt group of projective bundles\/}, Preprint.
	
	\bibitem{Witt}
	E. Witt, Theorie der quadratischen Formen in beliebigen K\"orpern. J. Reine
	Angew. Math. 176 (1937), 31–44.
	
	\bibitem{Panin}
	I. Panin, Riemann-Roch theorems for oriented cohomology, Axiomatic, enriched and motivic
	homotopy theory, NATO Sci. Ser. II 131 (2004) 261–333.
	
	\bibitem{Perrin}
	N. Perrin,
	Courbes rationnelles sur les vari\'et\'es homog\`enes, 
	Annales de l'Institut Fourier, Tome (1) 52 (2002), 105-132.
	
	\bibitem{Sch21}
	M. Schlichting, Higher K-theory of forms I. From rings to exact categories, J. Inst. Math. Jussieu. (4) 20 (2021), 1205-1273.
	
	\bibitem{QSS}
	H. G. Quebbemann, W. Scharlau, and M. Schulte,
	Quadratic and hermitian forms in additive and abelian
	categories,
	J. Algebra 59 (1979) 264-289.
	
	\bibitem{Rohrbach}
	H. Rohrbach,
	The Projective Bundle Formula for Grothendieck-Witt spectra,
	Journal of Pure and Applied Algebra (2021).
	
	
	\bibitem{Xie20}
	H. Xie, A transfer morphism for Hermitian K-theory of schemes with involution, Journal of Pure and Applied Algebra (4) 224 (2020), 26 pages.
	
	\bibitem{Xie19}
	H. Xie, Witt groups of smooth projective quadrics, Advances in Mathematics 346 (2019), 70-123.
	
	\bibitem{Zib11}
	M. Zibrowius,
	Witt groups of complex cellular varieties,
	Documenta Math. 16 (2011), 465–511.
	
	\bibitem{Zib14}
	M. Zibrowius,
	Witt groups of curves and surfaces,
	Math. Zeits. 278 (2014), Nr. 1–2, 191–227.
\end{thebibliography}
\end{document}